\documentclass{lms}
\usepackage{xcolor}
\usepackage{mathrsfs}
\usepackage{amssymb}
\usepackage{amscd}
\usepackage{appendix}

\usepackage{amsfonts}
\usepackage{amsmath}

\newtheorem{theorem}{Theorem}[section] 
\newtheorem{lemma}[theorem]{Lemma}     
\newtheorem{corollary}[theorem]{Corollary}

\newnumbered{assertion}{Assertion}    
\newnumbered{conjecture}{Conjecture}  
\newnumbered{definition}{Definition}
\newnumbered{hypothesis}{Hypothesis}
\newnumbered{remark}{Remark}
\newnumbered{note}{Note}
\newnumbered{observation}{Observation}
\newnumbered{problem}{Problem}
\newnumbered{question}{Question}
\newnumbered{algorithm}{Algorithm}
\newnumbered{example}{Example}
\newunnumbered{notation}{Notation}

\begin{document}
\noindent{\LARGE\bf Boundedness of semilinear Duffing equations at resonance with oscillating nonlinearities}
\vspace{0.5 true cm}\\
 \noindent{\normalsize Zhiguo Wang$^{1}$,\ Daxiong Piao$^{2}$ and  Yiqian Wang$^{3\dag}$

\footnotetext{\baselineskip 10pt
  \ \emph{2010 Mathematics Subject Classification}  34C15, 70H08.\\
 {\indent$^\dag$  The corresponding author}\\
 {\indent The first author  was supported by the National
Natural Science Foundation of China
 grant no. 11101299, 11271277. The second author was supported by the
Natural Science Foundation of Shandong Province grant no.
ZR2012AM018. The third author was supported by the National
Natural Science Foundation of China
 grant no. 11271183.} }} \vspace{0.2 true cm}
\renewcommand{\baselinestretch}{1.5}\baselineskip 12pt
\noindent{\footnotesize\rm  $^1$ School of Mathematical Sciences,
Soochow University, Suzhou
215006, P.R.China \\
(email: zgwang@suda.edu.cn )\vspace{4mm}\\
$^2$  School of Mathematical Sciences,  Ocean University of China,
Qingdao
266100, P.R.China\\
\footnotesize(email:dxpiao@ouc.edu.cn )\vspace{4mm}\\
$^3$ Department of Mathematics,
 Nanjing University, Nanjing 210093, P.R.China\\
\footnotesize(email: yiqianw@nju.edu.cn )}\\
\vspace{4mm}\\
\baselineskip 12pt \renewcommand{\baselinestretch}{1.18}

 \begin{abstract}  In this paper, we prove the  boundedness of all the solutions   for  the equation $\ddot{x}+n^2x+g(x)+\psi'(x)=p(t)$
 with the Lazer-Leach condition on $g$ and $p$, where $n\in \mathbb{N^+}$, $p(t)$ and $\psi'(x)$ are periodic and $g(x)$ is bounded. For the critical situation that $\big |\int_0^{2\pi}p(t)e^{int}dt \big
  |=2\big|g(+\infty)-g(-\infty)\big|$, we also prove a sufficient and necessary condition for the boundedness if $\psi'(x)\equiv0$.
\\
\\
\noindent {\it Keywords}: Hamiltonian system;\ boundedness;\ canonical transformation;
 \noindent {\ at resonance; \  oscillating nonlinearities;\ Moser's
 theorem.}
\end{abstract}
\section{Introduction and the main results}
 The study of semilinear equations at resonance has a long
history. The interest in this model is motivated both by its
connections to application and by a remarkable richness of the
related dynamical systems.

It is well known that the linear equation
$$\ddot{x}+n^2x=\sin nt,\ \ \ \quad \ n\in\mathbb{N^+} $$
has no bounded solutions, where $\displaystyle\ddot{x}=\frac{d^2x}{dt^2}$.
Another interesting example was constructed by Ding \cite{Ding},
who proved that each solution of the equation
$$\ddot{x}+n^2x+
\arctan x=4\cos nt,\ \ \ \quad \ n\in\mathbb{N^+} $$ is unbounded. Due to
these resonance phenomenons, the existence of bounded solutions
and the boundedness of all the solutions for semilinear equation at resonance
are very delicate.

In 1969, Lazer and Leach \cite{LL} studied the following
semilinear equations:
\begin{equation}\label{semi}
 \ddot{x}+n^2x+ g (x)=p(t),\
\ \ \quad \ n\in\mathbb{N^+}, \end{equation} where $p(t+2\pi)=p(t)$ and $g$ is
continuous and bounded. They proved that if \begin{equation}\label{L-L}
 \left|\int_0^{2\pi}p(t)e^{-int}dt\right|< 2(\liminf_{x\rightarrow +\infty}g-\limsup_{x\rightarrow -\infty}g),
\end{equation}
then (\ref{semi}) has at least one $2\pi$-periodic solution.
Moreover, they obtained that each solution of (\ref{semi}) is
unbounded if
$$
 \left|\int_0^{2\pi}p(t)e^{-int}dt\right|\ge 2(\sup g-\inf g).
$$
 Thus if \begin{equation}\label{monotone}\lim_{x\rightarrow-\infty} g(x)=g(-\infty)\le g(x)\le g(+\infty)=\lim_{x\rightarrow+\infty} g(x),\quad \forall x\in R,\end{equation}
then condition (\ref{L-L}) is sufficient and
necessary for the existence of bounded solutions. For this reason, (\ref{L-L}) is called Lazer-Leach condition.

In 1996, Alonso and Ortega \cite{Alo01} studied the
following equation: \begin{equation}\label{01}
\ddot{x}+n^2x+g(x)+\psi'(x)=p(t),\quad  n\in \mathbb{N^+}, \end{equation}
where $g$ and $p$ are as same as above and the perturbation
$\psi'(x)$ will be small at infinity in the following sense:
$$
\lim_{|x|\rightarrow \infty}\frac{\psi(x)}{x}=0.
$$ They proved that each solution with
large initial condition is unbounded if \begin{equation}\label{SN-condition}
\nonumber \left|\int_0^{2\pi}p(t)e^{-int}dt\right|> 2(H-K)
,\end{equation} where $$H=\max\{\limsup_{x\rightarrow -\infty}g,\quad
\limsup_{x\rightarrow +\infty}g\},\qquad
K=\min\{\liminf_{x\rightarrow -\infty}g,
\quad\liminf_{x\rightarrow +\infty}g\}.$$

Other conditions for the existence of bounded and unbounded
solutions are described in
\cite{Alo01,Alo02,GM,KM,Mawhin,MW} and
their references.

 The pioneering work on the boundedness of (\ref{semi}) was due to Ortega \cite{Ort02}. He proved a variant of Moser's small twist theorem, by which he obtained the boundedness  for the equation
$$
 \ddot{x}+n^2x+h_L(x)=p(t),\quad p(t)\in {\cal C}^5(\mathbb{R}/2\pi\mathbb{Z}),
$$
 where $L>0$ and $h_L(x)$ is of the form
 $$
 h_L(x)=\left\{\begin{array}{ll}
 L,& \quad {\rm if}\quad x\ge 1,\\
 Lx,&\quad {\rm if}\quad -1\le x\le 1,\\
 -L,& \quad {\rm if} \quad x\le -1,
 \end{array}\right.
 $$
and $p(t)$ satisfies
$$
\frac{1}{2\pi}\left|\int_0^{2\pi}p(t)e^{-int}dt\right|<\frac{2L}{\pi}.
$$

 Then Liu \cite{Liu01} studied the equation (\ref{semi}) by the assumptions:
 {$ p(t)\in {\cal C}^7(\mathbb{R}/2\pi\mathbb{Z}),$} $g(x)\in {\cal C}^6(\mathbb{R})$ satisfying
\begin{equation}\label{phi1}
 g(\pm\infty)=\lim_{x\rightarrow\pm\infty} g(x)\ {\rm\  exist\ and \ are\ finite,\ }\
\end{equation} and \begin{equation}\label{}
\nonumber \lim_{|x|\rightarrow +\infty}x^kg^{(k)}(x)=0,\ 0\leq k\leq6.
\end{equation}
 With Ortega's small twist theorem, he showed that the Lazer-Leach condition (\ref{L-L}) is sufficient for
 the boundedness of (\ref{semi}).
Moreover, if (\ref{monotone}) holds true, then Lazer-Leach's
result \cite{LL} implies that (\ref{L-L}) is also necessary for
the boundedness.

 One can refer to \cite{Liu01,Liu02,Liu03,Liu04,Ort02,Wangx} for the applications of Ortega's small twist theorem.

 In this paper, we study the boundedness of the equation
 \begin{equation}\label{01-1}
\ddot{x}+n^2x+g(x)+\psi'(x)=p(t),\quad  n\in \mathbb{N},
\end{equation} where
 $g(x)$ is bounded, $\psi(x+T)=\psi(x)$
and $p(t+2\pi)=p(t)$.

  We will prove that the Lazer-Leach condition (\ref{L-L}) on $g$ and $p$ is sufficient for the boundedness of (\ref{01-1}) with the existence of an oscillating term $\psi$.
  In other words, the oscillating term  does not play any role in the boundedness. More precisely,
 we prove that:

 \begin{theorem}\label{Th1}
Assume  $g(x)\in C^{\Upsilon_1}(\mathbb{R})$, $\psi'(x)\in C^{{\Upsilon_1}}(\mathbb{R}/T\mathbb{Z})$ and $p(t)\in C^{\Upsilon_2}(\mathbb{R}/2\pi\mathbb{Z})$ with $
\Upsilon_1=18,\ \Upsilon_2=14.$  Suppose the following two conditions hold true:
\begin{eqnarray}\label{}
  \nonumber(A_1)\ \ \ &&g(\pm \infty)=\lim_{x\rightarrow\pm\infty}g(x)\ \mbox{ exist},\\
 \nonumber && \lim_{|x|\rightarrow +\infty}x^kg^{(k)}(x)=0,\ 0\leq
  k\leq\Upsilon_1;\\
  \nonumber(A_2)\ \ \ && \int_0^T \psi(x)dx=0.
           \end{eqnarray}

 Then under the following Lazer-Leach condition:
\begin{equation}\label{LLc}
  {\Bigg |\int_0^{2\pi}p(t)e^{int}dt \Bigg
  |<2\big|g(+\infty)-g(-\infty)\big|,}
 \end{equation}every solution of (\ref{01-1})
 is bounded, i.e., for every $(t_0 , x_0 ,
\dot{x}_0),$ the solution $x(t; t_0 , x_0 , \dot{x}_0)$ exists for
all $t\in \mathbb{R}$ and it holds that $$sup_{ t\in\mathbb{R}
}\big( \big|x(t; t_0 , x_0 , \dot{x}_0)\big|+\big|\dot{x} (t; t_0
, x_0 , \dot{x}_0)\big| \big)<\infty.$$
\end{theorem}

\begin{remark}
 It is no loss of generality to assume the condition $(A_2)$, since if $\int_0^T \psi(x)dx\not=0$, we can make the transformation:
  $\tilde{\psi}(x)=\psi(x)-\frac{1}{T}\int_0^T \psi(x)dx.$
\end{remark}

 On the other hand, if
\begin{equation}
   {\Bigg |\int_0^{2\pi}p(t)e^{int}dt \Bigg
  |>2\big|g(+\infty)-g(-\infty)\big|,}
 \end{equation}
then Alonso-Ortega's result \cite{Alo01} implies the existence of unbounded solutions for (\ref{01-1}). Thus we obtain the following {conclusion}:
\begin{corollary}\label{cor1.2} Assume  $g(x),\ \psi(x)$ and $p(t)$ satisfy the conditions in Theorem \ref{Th1}. If
$\Bigg|\int_0^{2\pi}p(t)e^{int}dt
\Bigg|\not=2(g(+\infty)-g(-\infty))$, then (\ref{LLc}) is
sufficient and necessary for the boundedness of (\ref{01-1}).
\end{corollary}

For the {critical} situation that \begin{equation}\label{LL0}
 {\Bigg |\int_0^{2\pi}p(t)e^{int}dt \Bigg
  |=2\big|g(+\infty)-g(-\infty)\big|,}
 \end{equation}
 the only known result for  equation (\ref{semi}), see \cite{LL}, is for the case \begin{equation}\label{not-monotone}
\min\{g(-\infty), g(+\infty)\}\le  g(x)\le \max\{g(-\infty), g(+\infty)\}.
\end{equation}

In the following, we will consider the boundedness of (\ref{semi}) if $g$ does {\bf not} satisfy the condition (\ref{not-monotone}).

  {Suppose that $g(x)\in C^{\Upsilon_1}(\mathbb{R})$,
$p(t)\in C^{\Upsilon_2}(\mathbb{R}/2\pi\mathbb{Z})$  and there exist three constants $c_{\pm}>0$ and $0<d\not=1$ such that
\begin{equation}\label{G+-}
\lim_{|x|\rightarrow \pm\infty}x^{k-1+d}G_{\pm}^{(k)}(x)=0,\ 0<
  k\leq  {\Upsilon_1+1},
\end{equation}
where $G_{\pm}(x)=\int_0^x(g(x)-g(\pm\infty))dx-c_{\pm}\cdot(1+x^2)^{\frac{1-d}{2}}$.}

We have the following result:
\begin{theorem}\label{Th2}
 {Let   {$\Upsilon_1>5+\max\{4,7d, 28d-3\}$,
$ \Upsilon_2>1+\max\{4,7d, 28d-3\}$}, (\ref{phi1}),
(\ref{LL0}) and (\ref{G+-}) hold true.
Then the sufficient and necessary condition for the boundedness of (\ref{semi}) is $d<1$.}
\end{theorem}

\begin{remark}\label{nonmonotone}
From the definition of $G_{\pm}(x)$ and (\ref{G+-}), it is easy to see that (\ref{not-monotone})  does {not} hold true.
 Thus there is no contradiction with the result from \cite{LL}.
\end{remark}

 \begin{example}
Let $g(x)=\arctan x+2x(1+x^2)^{-\frac 23}$ and $p(t)=2\cos (nt)$.
Then the sufficient condition in Theorem \ref{Th2} {for boundedness} are met with
$g(\pm\infty)=\pm\frac{\pi}{2},\ d=\frac 23,\ c_{+}=3,\ c_{-}=3$. On the other hand, if $p(t)$ is kept unchanged except
that $g(x)$ is replaced by $\arctan x+\frac{x}{1+x^2}+x(1+x^2)^{-\frac 32}$, with  $d=\frac{3}{2}$ and $c_{\pm}= 1$ which implies (\ref{semi}) has unbounded solutions by Theorem \ref{Th2}. One can check that {these two functions do}  not satisfy {(\ref{not-monotone})}.
\end{example}


 \begin{remark}\label{more}
We can prove a similar result as the sufficient part in Theorem \ref{Th2} even if $\psi'(x)\not=0$ in the equation (\ref{01-1}). The proof is just a combination of the proof of Theorem \ref{Th1} and the one of Theorem \ref{Th2}.

 \end{remark}

 There are some new ingredients in our proof.

{Instead of applying Ortega's small twist theorem, we use a rotation transformation to deal with resonance (see also \cite{Xing}). With such a transformation, the linear term disappears in the new Hamiltonian (and a sublinear one is obtained), and one will not meet the difficulty of resonance any more. Then Moser's small twist theorem is directly applicable for the case $\psi(x)=0$, see \cite{Xing2012} and the proof of Theorem \ref{Th2}.}

{For the case $\psi(x)\not=0$, however, the rotation transformation is not sufficient for the study of boundedness. The reason lies in the fact that}  the potential in our equation does not {satisfy} the well-known polynomial growth condition due to the oscillating property of the function $\psi(x)$.  We say a bounded function $g(x,t)$ satisfies the polynomial growth condition with respect to $x$ if
  \begin{equation}\label{growthderivative}
 \lim_{|x|\rightarrow +\infty}x^mD_{x}^{m}g(x,t)=0
  \end{equation}
for some $m>0$.
 In most papers stated above,
the condition (\ref{growthderivative}) is required. Without the
polynomial growth condition, the estimates on the derivatives of
the perturbations are very poor. For this reason,
the perturbation in the sublinear system can not be reduced to be small enough in $C^4$-topology by
only repeated applications of the common method of generating function. Our observation is that although with the method of generating function the smoothness of the Hamiltonian on some variables (say, time variable) become worse, the one on some other variables (say, angle variable) become better and thus the poor
estimates on the corresponding variables can be improved, see Subsection \ref{4.3}. Starting from this key
observation, we can find further canonical transformations to
obtain a nearly integrable superlinear system and thus Moser's
theorem is available. It is worthy to note that the periodic assumption on $\psi(x)$ is not necessary. In fact we can  {show} the {boundedness holds} when $\psi(x)=\phi(x^{1+\delta})$ with $\phi(x)$ periodic and $\delta>0$ small enough. Moreover, $\psi(x)$ can be replaced by a function $\Psi(x,t)$ which is periodic on both $x$ and $t$, see \cite{Xing}. Thus we show that
 the classical polynomial growth conditions can be considerably weakened. For more references, one can see  \cite{lev1}, \cite{Wangy}.

 {It is well known that a sublinear system can be further changed into a superlinear one with the trick of exchanging the roles of time and angle {variables} (see \cite{Arn,lev1} for example). Thus  we conclude that for the boundedness of Duffing equations, there is no essential difference among semilinear, sublinear and superlinear cases.}


\vskip0.4cm
The paper is organized as follows.  The part from Section 2 to Section 5 is devoted to the proof of Theorem \ref{Th1}. In Section \ref{s2}, we state
some preliminary estimates. In Section \ref{s3}, we introduce a rotation transformation and then make canonical
transformations such that all non-oscillating terms  are
transformed into normal form possessing desirable properties. The
main difficulty in this paper lies in how to deal with oscillating
terms { caused by $\psi(x)$}. For this purpose, in Section \ref{s4} we make canonical transformations to
improve estimates on the derivatives of oscillating terms and subsequently change the system into a nearly integrable one.
Thus Theorem \ref{Th1} is proved by Moser's twist theorem in
Section \ref{s5}. {The sketch for the proof of Theorem \ref{Th2} is given in Section \ref{s6}.} The proof of some lemmas can be found in the Appendix.


\renewcommand{\baselinestretch}{1.1}\baselineskip 12pt

\section{Action-angle coordinates}\label{s2}
Consider the original system (\ref{01-1}).
 Let $y=\dot{x}/n$, equation (\ref{01-1}) is equivalent to a Hamiltonian system with   Hamiltonian
 \begin{equation}\label{H01}
    H(x,y,t)=\frac{1}{2}n(x^2+y^2)+\frac{1}{n}G(x)-\frac{1}{n}x
    p(t)+\frac{1}{n}\psi(x),
 \end{equation}
where $G(x)=\int_0^xg(s)ds.$

 Under the action-angle coordinates transformation($dx\wedge
dy=dI\wedge d\theta$)
  \begin{equation}\label{}\left\{ \begin{array}{l}
 \nonumber    x=x(I,\theta)=\sqrt{\frac{2}{n}}I^{\frac{1}{2}}\cos{n\theta}\\
  \nonumber     y=y(I,\theta)=\sqrt{\frac{2}{n}}I^{\frac{1}{2}}\sin{n\theta}
     \end{array}\right.,\ (I,\theta)\in\mathbb{R}^+\times
     {\mathbb{S}}^1,\ \mbox{with}\ {\mathbb{S}}^1=\mathbb{R}/(2\pi\mathbb{Z}),
      \end{equation}
  (\ref{H01}) is transformed into
 \begin{equation}\label{H02}
    H(I,\theta,t)=I+\frac{1}{n}G(x)-\frac{1}{n}x p(t)+\frac{1}{n}\psi(x),
 \end{equation}
 \ where\ $x=x(I,\theta)=\sqrt{\frac{2}{n}}I^{\frac{1}{2}}\cos{n\theta}$ for simplicity.

 Denote
$f_1(I,\theta)=\frac{1}{n}G(\sqrt{\frac{2}{n}}I^{\frac{1}{2}}\cos{n\theta}),\
f_2(I,\theta,t)=-\frac{1}{n}\sqrt{\frac{2}{n}}I^{\frac{1}{2}}\cos{n\theta}
p(t)$,\ and
$f_3(I,\theta)=\frac{1}{n}\psi(\sqrt{\frac{2}{n}}I^{\frac{1}{2}}\cos{n\theta})$,
then  (\ref{H02}) is rewritten by
  \begin{equation}\label{H03}
   \nonumber    H(I,\theta,t)=I+f_1(I,\theta)+f_2(I,\theta,t)+f_3(I,\theta).
 \end{equation}

\vskip0.4cm

In the context, we denote
$\displaystyle[f](\cdot)=\frac{1}{2\pi}\int_0^{2\pi}f(\cdot,\theta)d\theta$
be the average function of
 $f(\cdot,\theta)$ with respect to $\theta$. Without loss of generality, $C>1,\ c<1$ are two
 universal positive constants not concerning their quantities, and $j,k,l,\nu,\kappa$, etc., are non-negative integers.

 Next, we give several lemmas about the estimates on $f_1(I,\theta),\ f_2(I,\theta,t)$  and
 $f_3(I,\theta)$ which are similar to those in \cite{Liu01,Xing}.
 \begin{lemma}\label{L01}
 For $I$ large enough, $\theta\in \mathbb{S}^1$, $ k+j\leq\Upsilon_1+1,$ we have the estimates on $f_1(I,\theta)$ as following:
 $$cI^{\frac{1}{2}}\leq\Big|f_1(I,\theta)\Big|\leq CI^{\frac{1}{2}},\ \ \Big|\partial^k_I\partial^j_\theta f_1(I,\theta)\Big|\leq CI^{\frac{1}{2}-k+\frac{1}{2}(max\{1,j\}-1)};$$
 $$cI^{\frac{1}{2}}\leq \Big|[f_1](I)\Big|\leq CI^{\frac{1}{2}},\ \ cI^{\frac{1}{2}-k}\leq \Big| {[f_1]^{(k)}(I)}\Big|\leq CI^{\frac{1}{2}-k}.$$
  \end{lemma}

  \begin{lemma}\label{L02}
It holds that
 {$$\displaystyle\lim_{I\rightarrow+\infty}I^{-\frac{1}{2}}\cdot[f_1](I)=\frac{\sqrt{2}}{\pi}n^{-\frac{3}{2}}\big(g(+\infty)-g(-\infty)\big),$$
 $$\lim_{I\rightarrow+\infty}I^{\frac{1}{2}} \cdot [f_1]'(I)=\frac{\sqrt{2}}{2\pi}n^{-\frac{3}{2}}\big(g(+\infty)-g(-\infty)\big),$$
  $$\lim_{I\rightarrow+\infty}I^{\frac{3}{2}} \cdot [f_1]''(I)=-\frac{\sqrt{2}}{4\pi}n^{-\frac{3}{2}}\big(g(+\infty)-g(-\infty)\big).$$}
  \end{lemma}
\begin{remark}
 The estimates about Lemmas \ref{L01} and \ref{L02} are classic and can be obtained
 by {  direct calculations}. Thus we omit it. Readers can refer to
 \cite{Liu01}.
 \end{remark}
 Direct computations can lead to the following conclusions:
    \begin{lemma}\label{L03}
    For $I$ large enough, $\theta,\ t\in \mathbb{S}^1$, $ k+j\leq\Upsilon_1+1$ and $ l\leq {\Upsilon_2}$, we have the estimates on $f_2(I,\theta,t)$ as following:
 $$\Big|\partial^k_I\partial^j_\theta\partial^l_t f_2(I,\theta,t)\Big|\leq CI^{\frac{1}{2}-k}.$$
   \end{lemma}
     \begin{lemma}\label{L04}
 For $I$ large enough, $\theta\in \mathbb{S}^1$, $ k+j\leq\Upsilon_1+1,$  we have the estimates on $f_3(I,\theta)$ as following:
 \begin{equation}\label{f3001}
  \nonumber \Big|\partial^k_I\partial^j_\theta f_3(I,\theta)\Big|\leq
CI^{-\frac{k}{2}+\frac{j}{2}}.
 \end{equation}
    \end{lemma}

    \vskip 0.3cm


Since $\partial_I H>\frac{1}{2}$ when $I$ is
sufficiently large, we can solve $H(I,\theta,t)=h$ for $I$ as
following:
 \begin{equation}\label{I00}
   I= I(h,t,\theta)=h-R(h,t,\theta),
 \end{equation}
 where $R(h,t,\theta)$ is determined implicitly by the equation
 \begin{equation}\label{R00}
    R=f_1(h-R,\theta)+f_2(h-R,\theta,t)+f_3(h-R,\theta).
 \end{equation}
It is clear that $h\rightarrow+\infty$ if and only if $I\rightarrow+\infty$.
 { Meanwhile}, it is well known that the new Hamiltonian system
 \begin{equation}\label{I00-1}\left\{ \begin{array}{l}
  \nonumber \displaystyle  \frac{dt}{d\theta}=-\partial_{h}I(h,t,\theta),\\
   \\
   \nonumber    \displaystyle \frac{d h}{d\theta}=\partial_{ t}I(h,t,\theta)
     \end{array}\right.
      \end{equation}
       is equivalent to the original one, see \cite{Arn,lev1,Liu01,Xing}, etc.

      We present some estimates on $R(h,t,\theta)$ in (\ref{I00}).
\begin{lemma}\label{L05}
 For $h$ large enough, $\theta,\ t\in \mathbb{S}^1$, $ k+j\leq\Upsilon_1+1$ and $ l\leq {\Upsilon_2}$, we have the estimates on $R(h,t,\theta)$ as following:
 $$\Big|\partial^k_h\partial^l_t\partial^j_\theta  R\Big|\leq Ch^{\frac{1}{2}-\frac{k}{2}+\frac{1}{2}(max\{1,j\}-1)}.$$
   \end{lemma}
   The proof is given in the Appendix.
     \vskip0.4cm
 {Moreover, from { the identity} (\ref{R00})}, $R$ has the following form  by Taylor's
 {formula}:
 \begin{eqnarray}\label{R01}
   \nonumber R&=&f_1(h,\theta)+f_2(h,t,\theta)+f_3(h-R,\theta)\\
    \nonumber &&-\int_0^1\partial_If_1(h-\mu R,\theta) Rd\mu-\int_0^1\partial_If_2(h-\mu R,\theta,t) Rd\mu\\
   \nonumber &=&f_1(h,\theta)+f_2(h,\theta,t)+f_3(h-R,\theta)-\partial_If_1(h,\theta) R-\partial_If_2(h,\theta,t) R\\
              &&+\int_0^1\int_0^1\partial^2_If_1(h-s\mu R,\theta)\mu R^2dsd\mu+\int_0^1\int_0^1\partial^2_If_2(h-s\mu R,\theta,t)\mu
              R^2dsd\mu.
 \end{eqnarray}
 (\ref{R01}) yields that
  \begin{equation}\label{R03}
    R=f_1(h,\theta)+f_2(h,t,\theta)+\frac{1}{n}\psi(x)-R_{01}(h,t,\theta)-R_{02}(h,t,\theta),
 \end{equation}
 where
  \begin{equation}\label{f3}
   \nonumber   \frac{1}{n}\psi(x)=f_3(h-R,\theta),
 \end{equation}
 \begin{equation}\label{R001}
   \nonumber    R_{01}(h,t,\theta)=(\partial_If_1(h,\theta)+\partial_If_2(h,\theta,t))(f_1(h,\theta)+f_2(h,\theta,t)),
 \end{equation}
 and
 \begin{eqnarray}\label{R002}
   \nonumber R_{02}(h,t,\theta)&=&
  (\partial_If_1(h,\theta)+\partial_If_2(h,\theta,t))f_3(h-R,\theta)
              \\
   \nonumber &&+(\partial_If_1(h,\theta)+\partial_If_2(h,\theta,t))\left(     \int_0^1\partial_If_1(h-\mu R,\theta) Rd\mu+\int_0^1\partial_If_2(h-\mu R,\theta,t) Rd\mu\right) \\
   \nonumber              &&-\int_0^1\int_0^1\partial^2_If_1(h-s\mu R,\theta)\mu R^2dsd\mu-\int_0^1\int_0^1\partial^2_If_2(h-s\mu R,\theta,t)\mu R^2dsd\mu.
 \end{eqnarray}
 \begin{remark}
    In the above, we regard $-\frac{1}{n}\psi(x)$ as a composite function of new variables and we postpone the treatment of it in Section \ref{s4}.
       \end{remark}

Therefore, the   Hamiltonian is
  \begin{equation}\label{I00-1}
   \nonumber    I=h-f_1(h,\theta)-f_2(h,t,\theta)-\frac{1}{n}\psi(x)+R_{01}(h,t,\theta)+R_{02}(h,t,\theta),
 \end{equation}
 and the following estimates hold:

\begin{lemma}\label{L06}
 For $h$ large enough, $\theta,\ t\in \mathbb{S}^1$, $ k+j\leq\Upsilon_1-1,$ and $ l\leq {\Upsilon_2}$, it holds that:
 $$\Big|\partial^k_h\partial^l_t\partial^j_\theta R_{01}\Big|\leq Ch^{-k+\frac{1}{2}(max\{1,j\}-1)},$$
  and
  $$\Big|\partial^k_h\partial^l_t\partial^j_\theta R_{02}\Big|\leq Ch^{-\frac{1}{2}-\frac{k}{2}+\frac{1}{2}(max\{1,j\}-1)}.$$
    \end{lemma}
    The proof is given in the Appendix.

  \begin{remark}
    From
    Lemmas \ref{L01}, \ref{L02} and \ref{L06}, it shows that $f_1, f_2$ and $R_{01}$ satisfy the polynomial growth condition (\ref{growthderivative}) with variable $h$, while $-\frac{1}{n}\psi(x)$ and $R_{02}$
    do not satisfy the  polynomial growth condition due to the oscillating property of the periodic function $\psi(x)$.
       \end{remark}

     \section{The normal form of non-oscillating terms}\label{s3}
    In this section, we first introduce a rotation transformation to deal with resonance, then obtain the normal form for non-oscillating terms by canonical transformations.
    \subsection{A rotation transformation}
Define a rotation transformation
$\Phi_1:(h_1, t_1, \theta)\rightarrow(h, t, \theta)$ by
\begin{equation}\label{T01}\left\{ \begin{array}{l}
   \nonumber     h=h_1\\
   \nonumber      t=t_1+\theta.
     \end{array}\right.
 \end{equation}
Under $\Phi_1$, the Hamiltonian $I$ is transformed into $I_1$ as following
 \begin{eqnarray}\label{I01}
   \nonumber   {I_1(h_1,t_1,\theta)
   = -f_1(h_1,\theta)-f_2(h_1,\theta,t_1+\theta)-\frac{1}{n}\psi(x)+R_{11}(h_1,t_1,\theta)+R_{12}(h_1,t_1,\theta)
 }\end{eqnarray}
 with $R_{11}(h_1,t_1,\theta)=R_{01}(h_1,t_1+\theta,\theta),\ R_{12}(h_1,t_1,\theta)=R_{02}(h_1,t_1+\theta,\theta)$.

 \begin{lemma}\label{L07}
 For $h_1$ large enough, $\theta,\ t_1\in \mathbb{S}^1$, and $ k+j\leq\Upsilon_1-1,\  l\leq {\Upsilon_2},$ it holds that:

 $$|\partial^k_{h_1}\partial^l_{t_1}\partial^j_\theta R_{11}|\leq C{h_1}^{-k+\frac{1}{2}(max\{1,j\}-1)},$$
  and
  $$|\partial^k_{h_1}\partial^l_{t_1} \partial^j_\theta R_{12}|\leq C{h_1}^{-\frac{1}{2}-\frac{k}{2}+\frac{1}{2}(max\{1,j\}-1)}.$$
    \end{lemma}
\begin{proof}  It is obtained from Lemma \ref{L06}.\end{proof}

    \subsection{The normal form {with} $f_1(h_1,\theta)$}
    We make a canonical transformation $\Phi_2:(h_2,t_2,\theta)\rightarrow(h_1,t_1,\theta)$ given by
    \begin{equation}\label{T02}\left\{ \begin{array}{l}
     h_1=h_2,\\
      t_1=t_2-\partial_{ h_2}S_2(h_2,\theta)
     \end{array}\right.
 \end{equation}
 with the generating function $S_2(h_2,\theta)$ determined by
\begin{equation}\label{S02}
 S_2(h_2,\theta)=\int_0^\theta\big( f_1(h_2,\theta)-[f_1](h_2)\big)d\theta.
\end{equation}

Under $\Phi_2$, the Hamiltonian $I_1$ is transformed into $I_2$ as following
 \begin{eqnarray}\label{I02-1}
   \nonumber I_2(h_2,t_2,\theta)&=&-f_1(h_2,\theta)-f_2(h_2,\theta,t_2+\theta-\partial_{h_2}S_2(h_2,\theta))-\frac{1}{n}\psi(x)\\
  \nonumber  &&+R_{11}(h_2,t_2-\partial_{h_2}S_2(h_2,\theta),\theta)+R_{12}(h_2,t_2-\partial_{h_2}S_2(h_2,\theta),\theta)+\partial_{\theta}S_2(h_2,\theta)\\
 \nonumber \\
    \nonumber &=&-[f_1](h_2)-f_2(h_2,\theta,t_2+\theta)-\frac{1}{n}\psi(x) \\
    \nonumber &&+[f_1](h_2)-f_1(h_2,\theta)+\partial_{\theta}S_2(h_2,\theta)\\
    \nonumber &&+\int_0^1\partial_{t_1}f_2(h_2,\theta,t_2+\theta-\mu\partial_{h_2}S_2(h_2,\theta))\partial_{h_2}S_2(h_2,\theta)d\mu\\
   \nonumber  &&+R_{11}(h_2,t_2-\partial_{h_2}S_2(h_2,\theta),\theta)+R_{12}(h_2,t_2-\partial_{h_2}S_2(h_2,\theta),\theta).
 \end{eqnarray}
 It is clear that (\ref{S02}) implies
 \begin{equation}\label{}
  \nonumber  [f_1](h_2)-f_1(h_2,\theta)+\frac{\partial}{\partial \theta}S_2(h_2,\theta)=0.
\end{equation}
Let
\begin{eqnarray}\label{R021}
 \nonumber R_{21}(h_2,t_2,\theta)&=&R_{11}(h_2,t_2-\partial_{h_2}S_2(h_2,\theta),\theta)\\
  \nonumber  &&-\int_0^1\partial_{t_1}f_2(h_2,\theta,t_2-\mu\partial_{h_2}S_2(h_2,\theta))\partial_{h_2}S_2(h_2,\theta)d\mu,
 \end{eqnarray}
\begin{equation}\label{R022}
  \nonumber  R_{22}(h_2,t_2,\theta)\ =\ R_{12}(h_2,t_2-\partial_{h_2}S_2(h_2,\theta),\theta).\quad\quad\quad\quad\quad\quad\quad\quad\quad\quad
\end{equation}

Thus, $I_2$ is rewritten by
\begin{eqnarray}\label{I02}
   \nonumber    I_2(h_2,t_2,\theta)=-[f_1](h_2)-f_2(h_2,\theta,t_2+\theta)-\frac{1}{n}\psi(x)
    +R_{21}(h_2,t_2,\theta)+R_{22}(h_2,t_2,\theta).
 \end{eqnarray}

 We have the following estimates:
\begin{lemma}\label{L08}
 For $h_2$ large enough, $\theta,\ t_2\in \mathbb{S}^1$, it holds that
\begin{eqnarray}\label{S02-1}
 |\partial^k_{h_2}\partial^j_\theta S_2(h_2,\theta)|\leq
 Ch_2^{\frac{1}{2}-k+\frac{1}{2}(max\{2,j\}-2)},\   k+j\leq\Upsilon_1+1,
 \end{eqnarray}
 and
\begin{eqnarray}\label{phi2}
   \nonumber  && |\partial_{{h_2}} {t_1}|\leq C{h_2}^{-\frac{3}{2}},\ \
\partial_{{t_2}} {t_1}=1,\ \
|\partial_{\theta} {t_1}|\leq C{h_2}^{-\frac{1}{2}},\\
  \nonumber && |\partial^k_{{h_2}}\partial^l_{{t_2}}\partial^j_\theta
{t_1}|\leq C{h_2}^{-\frac{1}{2}-k+\frac{1}{2}(max\{2,j\}-2)},\
 k+l+j\geq2,\   k+j\leq \Upsilon_1.
 \end{eqnarray}
    Moreover, for $ k+j\leq \Upsilon_1-1$, it holds that:
   $$|\partial^k_{h_2}\partial^l_{t_2}\partial^j_\theta R_{21}|\leq C{h_2}^{-k+\frac{1}{2}(max\{1,j\}-1)},\  l\leq {\Upsilon_2}-1;$$
   $$ |\partial^k_{h_2}\partial^l_{t_2}\partial^j_\theta R_{22}|\leq C{h_2}^{-\frac{1}{2}-\frac{k}{2}+\frac{1}{2}(max\{1,j\}-1)},\  l\leq {\Upsilon_2}.$$

   \end{lemma}
    The proof is given in the Appendix.

   \vskip1cm

\subsection{The normal form {with} $f_2(h_2,\theta,t_2+\theta)$}

Without causing confusion, for convenience we still denote
 \begin{eqnarray}\label{}
    \nonumber   [f_2](h,t)&=&\frac{1}{2\pi}\int_0^{2\pi}f_2(h,\theta,t+\theta)d\theta.
 \end{eqnarray}
Then we have
  \begin{lemma}\label{LLc-1}
 For any $h\in\mathbb{R}^+$, $t\in \mathbb{S}^1$, it holds that
    \begin{eqnarray}\label{}
   \nonumber [f_2](h,t)&=&-\frac{\sqrt{2}}{2\pi}n^{-\frac{3}{2}}h^{\frac{1}{2}}\Big\{\cos
(nt)\int_0^{2\pi}p(\tau)\cos (n\tau)d\tau+\sin
(nt)\int_0^{2\pi}p(\tau)\sin( n\tau)d\tau\Big\}.
 \end{eqnarray}
 Moreover,
 \begin{eqnarray}\label{f2m}
    \Big|[f_2](h,t)\Big|&\leq&\frac{\sqrt{2}}{2\pi}n^{-\frac{3}{2}}h^{\frac{1}{2}}\Bigg|\int_0^{2\pi}p(\tau)e^{in\tau}d\tau\Bigg|.
 \end{eqnarray}

    \end{lemma}
   \begin{proof}
     \begin{eqnarray}\label{}
   \nonumber [f_2](h,t)&=&\frac{1}{2\pi}\int_0^{2\pi}f_2(h,\theta,t+\theta)d\theta\\
    \nonumber &=&-\frac{\sqrt{2}}{2\pi}n^{-\frac{3}{2}}h^{\frac{1}{2}}\int_0^{2\pi}\cos(n\theta)
p(t+\theta)d\theta\\
    \nonumber &=&-\frac{\sqrt{2}}{2\pi}n^{-\frac{3}{2}}h^{\frac{1}{2}}\{\cos(nt)\int_0^{2\pi}p(\tau)
  \nonumber \cos  (n\tau)d\tau+\sin(nt)\int_0^{2\pi}p(\tau)\sin( n\tau)d\tau\}.\\
\nonumber  \end{eqnarray}
 Thus, (\ref{f2m}) is obtained by the  norm of  complex number immediately.
   \end{proof}

Now, we make a transformation
$\Phi_3:(h_3,t_3, \theta)\rightarrow(h_2,t_2, \theta)$  implicitly given by
    \begin{equation}\label{T03}\left\{ \begin{array}{l}
     h_2=h_3+\partial_{ t_2}S_3(h_3,t_2,\theta)\\
      t_3\ =t_2+\partial_{ h_3}S_3(h_3,t_2,\theta)
     \end{array}\right.
 \end{equation}
 with the generating function $S_3(h_3,t_2,\theta)$ determined by
\begin{equation}\label{S03}
 S_3(h_3,t_2,\theta)=\int_0^\theta( f_2(h_3,\theta,t_2+\theta)-[f_2](h_3,t_2))d\theta.
\end{equation}

Under $\Phi_3$, the Hamiltonian $I_2$ is transformed into $I_3$ as
following
 \begin{eqnarray}\label{}
   \nonumber I_3(h_3,t_3,\theta)&=&-[f_1](h_3+\partial_{ t_2}S_3)-f_2(h_3+\partial_{ t_2}S_3,\theta,t_2+\theta)-\frac{1}{n}\psi(x)\\
     \nonumber&&+R_{21}(h_3+\partial_{ t_2}S_3,t_3-\partial_{ h_3}S_3,\theta)+R_{22}(h_3+\partial_{ t_2}S_3,t_3-\partial_{ h_3}S_3,\theta)+\partial_{\theta}S_3\\
         \nonumber &=&-[f_1](h_3)-[f_2](h_3,t_3)-\frac{1}{n}\psi(x) \\
    \nonumber &&+[f_2](h_3,t_2)-f_2(h_3,\theta,t_2+\theta)+\partial_{\theta}S_3\\
    \nonumber &&-\int_0^1[f_1]'(h_3+\mu\partial_{ t_2}S_3)\partial_{t_2}S_3(h_3,t_2,\theta)d\mu\\
     \nonumber  &&-\int_0^1\partial_If_2(h_3+\mu\partial_{ t_2}S_3,\theta,t_2+\theta)\partial_{t_2}S_3(h_3,t_2,\theta)d\mu\\
     \nonumber &&+\int_0^1\partial_t[f_2](h_3,t_3-\mu\partial_{ h_3}S_3)\partial_{h_3}S_3d\mu\\
       \nonumber    &&+R_{21}(h_3+\partial_{ t_2}S_3,t_3-\partial_{ h_3}S_3,\theta)+R_{22}(h_3+\partial_{ t_2}S_3,t_3-\partial_{ h_3}S_3,\theta).
 \end{eqnarray}
(\ref{S03}) implies
 \begin{equation}\label{}
  \nonumber  [f_2](h_3,t_2)-f_2(h_3,\theta,t_2+\theta)+\partial_{\theta}S_3=0.
\end{equation}
Let
 \begin{eqnarray}\label{}
   \nonumber \alpha(h_3,t_3)\ =-[f_1](h_3)-[f_2](h_3,t_3);\quad\quad\quad\quad\quad\quad\quad\quad
 \end{eqnarray}
 \begin{eqnarray}\label{R031}
   \nonumber R_{31}(h_3,t_3,\theta)&= &R_{21}(h_3+\partial_{ t_2}S_3,t_3-\partial_{ h_3}S_3,\theta)\\
  \nonumber  &&-\int_0^1[f_1]'(h_3+\mu\partial_{ t_2}S_3)\partial_{t_2}S_3(h_3,t_2,\theta)d\mu \\
    \nonumber   &&- \int_0^1\partial_If_2(h_3+\mu\partial_{ t_2}S_3,\theta,t_2+\theta)\partial_{t_2}S_3(h_3,t_2,\theta)d\mu\\
     \nonumber    &&+\int_0^1\partial_t[f_2](h_3,t_3-\mu\partial_{ h_3}S_3)\partial_{h_3}S_3d\mu;
 \end{eqnarray}
 \begin{eqnarray}\label{R032}
   \nonumber    R_{32}(h_3,t_3,\theta) &=&R_{22}(h_3+\partial_{ t_2}S_3,t_3-\partial_{
    h_3}S_3,\theta).\quad\quad\quad\quad\quad\quad\quad\quad\quad
 \end{eqnarray}
 Thus we have
  \begin{eqnarray}\label{I03}
   \nonumber I_3(h_3,t_3,\theta)=\alpha(h_3,t_3)-\frac{1}{n}\psi(x)
     +R_{31}(h_3,t_3,\theta)+R_{32}(h_3,t_3,\theta).
   \end{eqnarray}

\begin{lemma}\label{L09}
 For $h_3$ large enough, $\theta,\ t_3\in \mathbb{S}^1$, it holds that:
  $$\Big|\partial^k_{h_3}\partial^l_{t_2}\partial^j_\theta S_3(h_3,t_2,\theta)\Big|\leq Ch_3^{\frac{1}{2}-k},\  l\leq {\Upsilon_2},\ \forall\ k,j;$$
 \begin{eqnarray}\label{alpha}
ch_6^{\frac{1}{2}-k}\leq\Big|{\partial^k_{h_6}\alpha(h_6,t_6)}\Big|\leq
Ch_6^{\frac{1}{2}-k},\ k=0,1,2;
 \end{eqnarray}
  \begin{eqnarray}\label{alpha-1}
\Big|{\partial^k_{h_6}\partial^l_{t_6}\alpha(h_6,t_6)}\Big|\leq
Ch_6^{\frac{1}{2}-k},\ k\leq\Upsilon_1+1,\  l\leq {\Upsilon_2};
 \end{eqnarray}
  and for $ k+j\leq\Upsilon_1-1,$
  $$ \Big|\partial^k_{h_3}\partial^l_{t_3}\partial^j_\theta R_{31}\Big|\leq C{h_3}^{-k+\frac{1}{2}(max\{1,j\}-1)},\  l\leq {\Upsilon_2}-1;$$
  $$ \Big|\partial^k_{h_3}\partial^l_{t_3}\partial^j_\theta R_{32}\Big|\leq C{h_3}^{-\frac{1}{2}-\frac{k}{2}+\frac{1}{2}(max\{1,j\}-1)},\  l\leq {\Upsilon_2}-1.$$

     Moreover, the map $\Phi_3$ satisfies
\begin{eqnarray}\label{phi3}
 \nonumber  |\partial_{{h_3}} {t_2}|\leq C{h_3}^{-\frac{3}{2}},\ \
\frac{1}{2} \leq|\partial_{{t_3}} {t_2}|\leq 2,\ \
|\partial_{\theta} {t_2}|\leq C{h_3}^{-\frac{1}{2}},\\
\nonumber |\partial^k_{{h_3}}\partial^l_{{t_3}}\partial^j_\theta
{t_2}|\leq C{h_3}^{-\frac{1}{2}-k},\ k+l+j\geq2,\  l\leq
{\Upsilon_2};
\\
\nonumber\frac{1}{2} \leq|\partial_{{h_3}} {h_2}|\leq 2,\ \
|\partial_{{t_3}} {h_2}|\leq C{h_3}^{\frac{1}{2}},\ \
|\partial_{\theta} {h_2}|\leq C{h_3}^{\frac{1}{2}},
\\
\nonumber |\partial^k_{{h_3}}\partial^l_{{t_3}}\partial^j_\theta
{h_2}|\leq C{h_3}^{\frac{1}{2}-k},\ k+l+j\geq2,\  l\leq
{\Upsilon_2}-1.
 \end{eqnarray}
  \end{lemma}
\begin{proof}  From (\ref{LLc}),  Lemmas \ref{L02} and
 \ref{LLc-1}, (\ref{alpha}) and (\ref{alpha-1}) holds.
 The rest of the proof  is similar to the one of lemma \ref{L08}.\end{proof}

   \begin{remark}
    For the case $\psi(x)=0$, it is not difficult to obtain the boundedness of the system with Hamiltonian $I_3$ by some canonical transformations, see \cite{Xing2012}. However if  $\psi(x)\not=0$, the perturbation of $I_3$ does not satisfy the polynomial growth condition and thus further canonical transformations are needed.
   \end{remark}

\section{The oscillating terms}\label{s4}
The main difficulty of this paper is how to deal with the oscillating {terms caused by $\psi(x)$}. Without the polynomial growth condition, the estimates on the
derivatives of the oscillating terms  are very poor. For this reason,
we cannot reduce the perturbation in the sublinear system to be a
small one by only repeated applications of canonical transformations.
A key observation is that the canonical transformations can help to improve
the poor estimates, see Subsection \ref{4.3}. Thus we can find further canonical transformations to
obtain a nearly integrable superlinear system, see Subsection
\ref{4.5},  then Moser's theorem is available.

\subsection{A canonical transformation for
$ \psi(x)$}\label{s4.1}

In this subsection, we will make a transformation
 to deal with
$ \psi(x)$. Recall all the transformations we have done
before this section:
 $$(x,y,t)\rightarrow(I,\theta,t), \ \mbox{where}\ x=\sqrt{\frac{2}{n}}I^{\frac{1}{2}}\cos{n\theta},\ y=\sqrt{\frac{2}{n}}I^{\frac{1}{2}}\sin{n\theta};$$
 $$(I,\theta,t)\rightarrow(h,t,\theta), \ \mbox{where}\ I=h-R(h,t,\theta);$$
 and then
 $$(h,t, \theta)=\Phi_1(h_1,t_1, \theta),\ (h_1,t_1, \theta)=\Phi_2(h_2,t_2, \theta),\ (h_2,t_2, \theta)=\Phi_3(h_3,t_3, \theta).$$

   Thus
   \begin{eqnarray}\label{}
   \nonumber  \psi(x)= \psi \Big(\sqrt{\frac{2}{n}}I^{\frac{1}{2}}\cos{n\theta}\Big)= \psi
 \Big(\sqrt{\frac{2}{n}}(h-R)^{\frac{1}{2}}\cos{n\theta}\Big)=\cdots\\
  \nonumber  = \psi
\Big
(\sqrt{\frac{2}{n}}(h_3)^{\frac{1}{2}}(1+Q(h_3,t_3,\theta))^{\frac{1}{2}}\cos{n\theta}\Big),
   \end{eqnarray}
where
$Q(h_3,t_3,\theta)=h_3^{-1}\big(h_2(h_3,t_3,\theta)-h_3-R(h_2(h_3,t_3,\theta),t(h_3,t_3,\theta),\theta)\big)$
satisfies
\begin{lemma}\label{LQ1}
 For $h_3$ large enough, $\theta,\ t_3\in \mathbb{S}^1$, $ k+j\leq\Upsilon_1,$ and  $ l\leq {\Upsilon_2}-1$, it holds that:
 \begin{equation}\label{Q001}
 \Big|\partial^k_{h_3}\partial^l_{t_3}\partial^j_{\theta}Q(h_3,t_3,\theta)\Big|\leq C{h_3}^{-\frac{1}{2}-\frac{k}{2}+\frac{1}{2}(max\{1,j\}-1)}.
 \end{equation}
Moreover, the following equation holds true:
\begin{eqnarray}\label{Q02}
\partial^2_\theta
 Q(h_3,t_3,\theta)=q_1(h_3,t_3,\theta)+q_2(h_3,t_3,\theta)\sin^2 n\theta
 \end{eqnarray}
with $|q_1|\leq Ch_3^{-\frac{1}{2}},\ |q_2|\leq C$.
  \end{lemma}
  The proof is technical and we give it in the Appendix.

   \vskip1cm

 For convenience, we denote
$\widetilde{f}_3(h_3, t_3,\theta)=\frac{1}{n}\psi(x).$ Using Pan
and Yu's method \cite{Pan}, we can prove that { the average}
$[\widetilde{f}_3](h_3, t_3)$ possesses an estimate better than
the one for $\tilde{f}_3$ itself as follows.
\begin{lemma}\label{L10}
 For $h_3$ large enough, $\theta,\ t_3\in \mathbb{S}^1$, it holds that
 \begin{equation}\label{pan01}
 \Big|[\widetilde{f}_3](h_3, t_3)\Big|=\Big|\int_0^{2\pi}f_3(h_3+\partial_{
t_2}S_3,\theta)d\theta\Big|\leq C{h_3}^{-\frac{1}{4}};
 \end{equation}
moreover,
\begin{equation}\label{pan02}
\Big|\partial^k_{h_3}\partial^l_{t_3}[\widetilde{f}_3](h_3,
t_3)\Big|\leq C{h_3}^{-\frac{1}{4}-\frac{k}{2}},\
l\leq{\Upsilon_2}-1,\  k\leq\Upsilon_1.
 \end{equation}
  \end{lemma}
\begin{proof}
Note that  $\psi(x)\in C^{{\Upsilon_1+1}}(\mathbb{R}/(T\mathbb{Z}))$ and $\int_0^T
\psi(x)dx=0$ by the assumption ($A_2$), it follows that
\begin{eqnarray}\label{psi-1}
   \psi(x)=\sum_{m=1}^{+\infty}(\psi_{m}^1\sin \frac{2m\pi}{T}x+\psi_{m}^2\cos \frac{2m\pi}{T}x),
   \end{eqnarray}
where the Fourier coefficients satisfy,   {integrated by parts},
\begin{eqnarray}\label{psi-2}
  \Big|\psi_{m}^1\Big|=\Big|\frac{2}{T}\int_0^T\psi(x)\sin \frac{2m\pi}{T}x dx\Big|\leq m^{-{\Upsilon_1}-1},
   \end{eqnarray}
   \begin{eqnarray}\label{psi-3}
   \Big|\psi_{m}^2\Big|=\Big|\frac{2}{T}\int_0^T\psi(x)\cos \frac{2m\pi}{T}x dx\Big|\leq
   m^{-{\Upsilon_1}-1}.
  \end{eqnarray}

For given $m$,
 consider the estimates on $$\int_0^{2\pi}\sin (\frac{2m\pi}{T}x)d\theta=\int_0^{2\pi}\sin
\Big(\frac{2m\pi}{T}
 \sqrt{\frac{2}{n}}h_3^{\frac{1}{2}}u(h_3,t_3,\theta)\cos{n\theta}\Big)d\theta$$
 with $u=(1+Q(h_3,t_3,\theta))^{\frac{1}{2}}$.
\vskip0.3cm
 \emph{ Step 1.}\quad Note that $\frac{1}{2}\leq\Big|u(h_3,t_3,\theta)\Big|<C$.
 And for $ k+j\leq\Upsilon_1,\ l\leq{\Upsilon_2}-1$,    $k+l+j\geq1,$ we have
 $$\Big|\partial^k_{h_3}\partial^l_{t_3}\partial^j_\theta u(h_3,t_3,\theta)\Big|\leq C{h_3}^{-\frac{1}{2}-\frac{k}{2}+\frac{1}{2}(max\{1,j\}-1)}
$$ by Leibniz's rule and
Lemma \ref{LQ1}.

Moreover, for $ k\leq\Upsilon_1,\ l\leq{\Upsilon_2}-1$ and
$k+l\geq1,$ it holds that
  \begin{equation}\label{pan03}
  \Big|\partial^k_{h_3}\partial^l_{t_3}\sin \frac{2m\pi}{T}x\Big|\leq C{h_3}^{-\frac{k}{2}}.
 \end{equation}
\vskip0.3cm
\emph{Step 2.}\quad Let
 $v(h_3,t_3,\theta)=\frac{2\pi}{T}\sqrt{\frac{2}{n}}u(h_3,t_3,\theta)\cos{n\theta},$
 then
 $$\partial_\theta v(h_3,t_3,\theta)=\frac{2\pi}{T}\sqrt{\frac{2}{n}}\{\partial_\theta
 u(h_3,t_3,\theta)\cos{n\theta}-nu(h_3,t_3,\theta)\sin{n\theta}\},$$
 $$\partial^2_\theta v(h_3,t_3,\theta)=\frac{2\pi}{T}\sqrt{\frac{2}{n}}\{\partial^2_\theta
 u(h_3,t_3,\theta)\cos{n\theta}-2n\partial_\theta u(h_3,t_3,\theta)\sin{n\theta}-n^2u(h_3,t_3,\theta)\cos{n\theta}\}.$$
 Note that
  $$\partial_\theta
 u(h_3,t_3,\theta)=\frac{1}{2}(1+Q)^{-\frac{1}{2}}\partial_\theta Q.$$
 For fixed $h_3,t_3$, assume  $\theta^*$ is a critical point of
 $v(h_3,t_3,\theta)$, i.e. $\partial_\theta v(h_3,t_3,\theta^*)=0$. Then from (\ref{Q001}), we have
 $\sin n\theta^*\rightarrow 0\ $ and $\cos n\theta^*\rightarrow 1\  as\ h_3\rightarrow\infty$. To prove that
 $\theta^*$ is an isolated critical point, we consider $$\partial^2_\theta
 u(h_3,t_3,\theta)=-\frac{1}{4}(1+Q)^{-\frac{3}{2}}(\partial_\theta Q)^2+\frac{1}{2}(1+Q)^{-\frac{1}{2}}\partial^2_\theta Q.$$
 Thus it follows from (\ref{Q02}) that
 $$\partial^2_\theta
 u(h_3,t_3,\theta)\cos{n\theta}=q_3(h_3,t_3,\theta)+q_4(h_3,t_3,\theta)\sin^2 n\theta
$$
with $|q_3|\leq Ch_3^{-\frac{1}{2}},\ |q_4|\leq C$.

 Therefore it shows that $\partial^2_\theta v(h_3,t_3,\theta^*)<-\frac{\sqrt{2}\pi}{T}n^{\frac{3}{2}}\neq0.$
 On the other hand, it is easy to see the existence of such critical points. In conclusion, we have shown that for given $(h_3,t_3)$, $v(h_3,t_3,\theta)$ has  finitely many isolated critical
 points in the interval $\theta\in[0,2\pi]$.
\vskip0.3cm
\emph{ Step 3.}\quad
  Without loss of generality, for given $(h_3,t_3)$, suppose  $[a,b]\subset[0,2\pi]$ is an interval where $a,\ b$
 are
 the only two critical points of $v(h_3,t_3,\theta)$.
 Following Pan and Yu's method (see Lemma \ref{A03} and Remark \ref{A03-1} in the Appendix),
 with $\lambda=\mu=1,\ \rho=\sigma=2,\ \nu=2,$
 we have, for $mh_3^{\frac{1}{2}}\gg1,$
 \begin{eqnarray}\label{sin-1}
   \nonumber  \int_a^be^{imh_3^{\frac{1}{2}}v(h_3,t_3,\theta)}d\theta&\sim&B_1(mh_3^{\frac{1}{2}})-A_1(mh_3^{\frac{1}{2}}),
 \end{eqnarray}
 where
 $$A_1(mh_3^{\frac{1}{2}})=-C_1e^{i(mh_3^{\frac{1}{2}}v(h_3,t_3,a)+{\frac{\pi}{4}})}m^{-\frac{1}{2}}h_3^{-\frac{1}{4}};$$
 $$B_1(mh_3^{\frac{1}{2}})=-C_2e^{i(mh_3^{\frac{1}{2}}v(h_3,t_3,b)+{\frac{\pi}{4}})}m^{-\frac{1}{2}}h_3^{-\frac{1}{4}}$$
with $C_1, C_2$ independent of $m$.

 Then we have
\begin{eqnarray}\label{psi-4}
 \Big |\int_a^b\sin(mh_3^{\frac{1}{2}}v(h_3,t_3,\theta))d\theta \Big|\leq Cm^{-\frac{1}{2}}h_3^{-\frac{1}{4}}.
 \end{eqnarray}

In the same way, we can prove
\begin{eqnarray}\label{psi-5}
  \Big |\int_a^b\cos(mh_3^{\frac{1}{2}}v(h_3,t_3,\theta))d\theta \Big|\leq Cm^{-\frac{1}{2}}h_3^{-\frac{1}{4}}.
 \end{eqnarray}

Together with (\ref{psi-1})-(\ref{psi-3}), (\ref{psi-4}) and (\ref{psi-5}), we obtain that
\begin{eqnarray}\label{}
 \nonumber\Big|[\widetilde{f}_3](h_3, t_3)\Big|&=&\frac{1}{n}\Big|\int_0^{2\pi}\psi(x)d\theta\Big|\\
\nonumber&\leq&\frac{1}{n}\sum_{m=1}^{+\infty}\Big(|\psi_{m}^1|\cdot\Big|\int_0^{2\pi}\sin \frac{2m\pi}{T}xd\theta\Big|+|\psi_{m}^2|\cdot\Big|\int_0^{2\pi}\cos \frac{2m\pi}{T}xd\theta\Big|\Big)\\
   \nonumber &\leq& C\sum_{m=1}^{+\infty}m^{-{\Upsilon_1}-\frac{1}{2}}h_3^{-\frac{1}{4}}\leq C{h_3}^{-\frac{1}{4}}.
 \end{eqnarray}

 Hence (\ref{pan01}) is proved.
 \vskip0.3cm

To prove (\ref{pan02}), note that
\begin{equation}\label{}
  \nonumber I_{kl}=\partial^k_{h_3}\partial^l_{t_3}[\widetilde{f}_3](h_3,
t_3)=\int_0^{2\pi}\partial^k_{h_3}\partial^l_{t_3}f_3(h_3+\partial_{
t_2}S_3,\theta)d\theta
 \end{equation}
with
 $\partial^k_{h_3}\partial^l_{t_3}f_3(h_3+\partial_{
t_2}S_3,\theta)=\frac{1}{n}\partial^k_{h_3}\partial^l_{t_3}\psi
 (\sqrt{\frac{2}{n}}h_3^{\frac{1}{2}}u(h_3,t_3,\theta)\cos{n\theta})$.
From (\ref{psi-1}), it follows that
$$\partial^k_{h_3}\partial^l_{t_3}\psi
 (x)= \sum_{m=1}^{+\infty}(\psi_{m}^1\partial^k_{h_3}\partial^l_{t_3}\sin \frac{2m\pi}{T}x+\psi_{m}^2\partial^k_{h_3}\partial^l_{t_3}\cos \frac{2m\pi}{T}x)$$
 with $x=\sqrt{\frac{2}{n}}h_3^{\frac{1}{2}}u(h_3,t_3,\theta)\cos{n\theta}.$

 By Leibniz's rule, each term $\partial^k_{h_3}\partial^l_{t_3}\sin \frac{2m\pi}{T}x$ is of the form
$$
\varphi^1_{mkl}(h_3,t_3,\theta)\cdot\sin
 \frac{2m\pi}{T}x+
\varphi^2_{mkl}(h_3,t_3,\theta)\cdot\cos
 \frac{2m\pi}{T}x.$$
Form (\ref{pan03}),  for $ k\leq\Upsilon_1,\ l\leq{\Upsilon_2}-1$, it holds that
  \begin{equation}\label{}
   \nonumber  \Big|\varphi^i_{mkl}(h_3,t_3,\theta)\Big|\leq C{h_3}^{-\frac{k}{2}},\
  i=1,2.
 \end{equation}

  Let $\varphi^i(\theta)=\frac{1}{n}{h_3}^{\frac{k}{2}}\varphi^i_{mkl}(h_3,t_3,\theta)$ for $i=1, 2$,
 then
 \begin{eqnarray}\label{}
  \nonumber  I_{mkl}&=&\int_0^{2\pi}\partial^k_{h_3}\partial^l_{t_3}\sin \frac{2m\pi}{T}xd\theta\\
    \nonumber   &=&{h_3}^{-\frac{k}{2}}\int_0^{2\pi}\big(\varphi^1(\theta)\cdot\sin
 \frac{2m\pi}{T}x+\varphi^2(\theta)\cdot\cos
 \frac{2m\pi}{T}x\big)d\theta.
 \end{eqnarray}
 Then repeating Step 1,\ Step 2  \& Step 3 above,  with the help of Lemma \ref{A03},
 we obtain (\ref{pan02}).
 \end{proof}





 Now  we make a
transformation $\Phi_4:(h_4,t_4,\theta)\rightarrow(h_3,t_3,\theta)$ implicitly
given
 by
    \begin{equation}\label{T04}\left\{ \begin{array}{l}
     \nonumber   h_3=h_4+\partial_{ t_3}S_4(h_4,t_3,\theta)\\
    \nonumber     t_4\ =t_3+\partial_{ h_4}S_4(h_4,t_3,\theta)
     \end{array}\right.
 \end{equation}
 with the generating function $S_4(h_4,t_3,\theta)$ determined by
\begin{equation}\label{S04}
  \nonumber  S_4(h_4,t_3,\theta)=\int_0^\theta(\widetilde{f}_3(h_4, t_3,\theta) -[\widetilde{f}_3](h_4, t_3))d\theta.
\end{equation}
Under $\Phi_4$, the Hamiltonian $I_3$ is transformed into $I_4$ as
following
 { {\begin{eqnarray}\label{I04-2}
   \nonumber I_4(h_4,t_4,\theta)&=&\alpha(h_4+\partial_{ t_3}S_4,t_4-\partial_{ h_4}S_4)-\widetilde{f}_3(h_4+\partial_{ t_3}S_4,t_3,\theta)\\
     \nonumber&&+R_{31}(h_4+\partial_{ t_3}S_4,t_4-\partial_{ h_4}S_4,\theta)+R_{32}(h_4+\partial_{ t_3}S_4,t_4-\partial_{ h_4}S_4,\theta)+\partial_{\theta}S_4\\
      \nonumber     &=&\alpha(h_4,t_4)-[\widetilde{f}_3](h_4,t_4)
     +R_{41}(h_4,t_4,\theta)+R_{42}(h_4,t_4,\theta)+R_{43}(h_4,t_4,\theta),
 \end{eqnarray}
where
 \begin{eqnarray}\label{R041}
    \nonumber  R_{41}(h_4,t_4,\theta)&=&R_{31}(h_4,t_4,\theta);\quad\quad\quad\quad\quad\quad\quad\quad\quad\quad\quad\quad\quad\quad\quad
 \end{eqnarray}
 \begin{eqnarray}\label{R042}
   \nonumber  R_{42}(h_4,t_4,\theta)&=&\int_0^1\partial_I\alpha(h_4+\mu\partial_{ t_3}S_4,t_3)\partial_{t_3}S_4(h_4,t_3,\theta)d\mu\\
          \nonumber  && -\int_0^1\int_0^1\partial^2_t\alpha(h_4,t_4-s\mu\partial_{
          h_4}S_4)\mu(\partial_{h_4}S_4)^2dsd\mu\\
           \nonumber  &&-\int_0^1\partial_{h_3}\widetilde{f}_3(h_4+\mu\partial_{ t_3}S_4,t_3,\theta)\partial_{t_3}S_4(h_4,t_3,\theta)d\mu\\
    \nonumber &&+\int_0^1\partial_{t_3}[\widetilde{f}_3](h_4,t_4-\mu\partial_{ h_4}S_4)\partial_{h_4}S_4d\mu\\
       \nonumber  &&+\int_0^1\partial_{h_3} R_{31}(h_4+\mu\partial_{ t_3}S_4,t_3,\theta)\partial_{t_3}S_4(h_4,t_3,\theta)d\mu\\
     \nonumber &&-\int_0^1\partial_{t_3} R_{31}(h_4,t_4-\mu\partial_{ h_4}S_4,\theta)\partial_{h_4}S_4d\mu\\
    \nonumber    &&+R_{32}(h_4+\partial_{ t_3}S_4,t_4-\partial_{
     h_4}S_4,\theta);
 \end{eqnarray}
  \begin{eqnarray}\label{R043}
     R_{43}(h_4,t_4,\theta)&=&
     -\partial_t\alpha(h_4,t_4)\cdot\partial_{h_4} S_4.\quad\quad\quad\quad\quad\quad\quad\quad\quad\quad\quad\quad
     \end{eqnarray}
 }

We have the following estimates.
\begin{lemma}\label{L11}
 For $h_4$ large enough, $\theta,\ t_4\in \mathbb{S}^1$ and $l\leq{\Upsilon_2}-1,$ it holds that
 $$\Big|\partial^k_{h_4}\partial^l_{t_3} S_4(h_4,t_3,\theta)\Big|\leq Ch_4^{-\frac{1}{4}-\frac{k}{2}},\  k+j\leq\Upsilon_1,$$
 $$ \Big|\partial^k_{h_4}\partial^l_{t_3}\partial^j_\theta S_4(h_4,t_3,\theta)\Big|\leq Ch_4^{-\frac{k}{2}+\frac{j-1}{2}},\ j\geq1,\  k+j\leq\Upsilon_1;$$
and for $l\leq{\Upsilon_2}-2$,
 $$ \Big|\partial^k_{h_4}\partial^l_{t_4}\partial^j_\theta R_{41}\Big|\leq C{h_4}^{-k+\frac{1}{2}(max\{1,j\}-1)},\  k+j\leq\Upsilon_1-1 ;$$
  $$  \Big|\partial^k_{h_4}\partial^l_{t_4}\partial^j_\theta R_{42}\Big|\leq C{h_4}^{-\frac{1}{2}-\frac{k}{2}+\frac{1}{2}(max\{1,j\}-1)},\  k+j\leq\Upsilon_1-2;$$
     $$ \Big|\partial^k_{h_4}\partial^l_{t_4} R_{43}\Big|\leq C{h_4}^{-\frac{1}{4}-\frac{k}{2}},\ \Big|\partial^k_{h_4}\partial^l_{t_4}\partial^j_\theta R_{43}\Big|\leq C{h_4}^{-\frac{k}{2}+\frac{j-1}{2}},\  k+j\leq\Upsilon_1-2.$$
    \end{lemma}
    \begin{proof} The estimates on $R_{41},\ R_{42}$ are similar to those in Lemma \ref{L08}.
    Note that
    \begin{equation}\label{}
  \nonumber |  S_4(h_4,t_3,\theta)|=\Big|\int_0^\theta(\widetilde{f}_3(h_4,
t_3,\theta) -[\widetilde{f}_3](h_4, t_3))d\theta\Big|,
\end{equation}
and
   \begin{equation}\label{}
  \nonumber |\partial_\theta S_4(h_4,t_3,\theta)|\leq C|\partial_\theta\widetilde{f}_3(h_4, t_3,\theta)|,
\end{equation}
thus the estimates on $S_4$ are obtained from Lemma \ref{L10}.
Finally, the estimates on $R_{43}$ can be obtained directly from
(\ref{R043}).\end{proof}

The oscillating terms in $I_4$ include $-[\widetilde{f}_3]$, $R_{42}$ and $R_{43}$ {while the worst term among them is $R_{43}$}. For simplicity,
without causing confusion, we still denote the sum
$-[\widetilde{f}_3]+R_{42}+R_{43}$ by $R_{43}$, i.e.
\begin{eqnarray}\label{I04-1}
  \nonumber  I_4(h_4,t_4,\theta)&=&\alpha(h_4,t_4)+R_{41}(h_4,t_4,\theta)+R_{43}(h_4,t_4,\theta).
 \end{eqnarray}
    \vskip1cm

\subsection{A canonical transformation for $R_{41}$}\label{4.2}
 Before dealing with the oscillating term $R_{43}$, we first reduce the non-oscillating term $R_{41}$ to be small enough
  by a canonical transformation.

Let $\Phi_5:(h_5,t_5,\theta)\rightarrow(h_4,t_4,\theta)$ be implicitly given by
    \begin{equation}\label{T05}\left\{ \begin{array}{l}
    \nonumber   h_4=h_5+\partial_{ t_4}S_5(h_5,t_4,\theta)\\
    \nonumber    t_5\ =t_4+\partial_{ h_5}S_5(h_5,t_4,\theta)
     \end{array}\right.
 \end{equation}
 with the generating function $S_5(h_5,t_4,\theta)$ determined by
\begin{equation}\label{S05}
   \nonumber S_5(h_5,t_4,\theta)=-\int_0^\theta(R_{41}(h_5, t_4,\theta) -[R_{41}](h_5, t_4))d\theta.
\end{equation}
Under $\Phi_5$, the Hamiltonian $I_4$ is transformed into $I_5$ as
following
 { \begin{eqnarray}\label{I05-2}
   \nonumber I_5(h_5,t_5,\theta)&=&\alpha(h_5+\partial_{ t_4}S_5,t_5-\partial_{ h_5}S_5)+R_{43}(h_5+\partial_{ t_4}S_5,t_5-\partial_{ h_5}S_5,\theta)\\
      \nonumber&&+R_{41}(h_5+\partial_{ t_4}S_5,t_4,\theta)+\partial_{\theta}S_5\\
                 &=&\alpha(h_5,t_5)+[R_{41}](h_5,t_5)
    +R_{5}(h_5,t_5,\theta),
 \end{eqnarray}
where
  \begin{eqnarray}\label{R05}
   \nonumber  R_{5}(h_5,t_5,\theta)&=&\int_0^1\partial_I\alpha(h_5+\mu\partial_{ t_4}S_5,t_4)\cdot\partial_{t_4}S_5d\mu-\int_0^1\partial_t\alpha(h_5,t_5-\mu\partial_{ h_5}S_5)\cdot\partial_{h_5}S_5d\mu\\
     \nonumber &&+\int_0^1\partial_{t_4}[R_{41}](h_5,t_5-\mu\partial_{ h_5}S_5)\cdot\partial_{h_5}S_5d\mu+R_{43}(h_5+\partial_{ t_4}S_5,t_5-\partial_{ h_5}S_5,\theta).
 \end{eqnarray}}

We have the following estimates:
\begin{lemma}\label{L12}
 For $h_5$ large enough, $\theta, \ t_5\in \mathbb{S}^1$,  we have the estimates on $S_5(h_5,t_4,\theta)$, $[R_{41}](h_5,t_5)$,  and $R_{5}(h_5,t_5,\theta)$ as following: for $ k+j\leq\Upsilon_1-1,\ l\leq{\Upsilon_2}-2,$
 $$\Big|\partial^k_{h_5}\partial^l_{t_4}\partial^j_\theta S_5(h_5,t_4,\theta)\Big|\leq Ch_{5}^{-k+\frac{1}{2}(max\{2,j\}-2)},$$
 $$ \Big|\partial^k_{h_5}\partial^l_{t_5} [R_{41}](h_5,t_5)\Big|\leq C{h_5}^{ {-k}};$$
   and for $ k+j\leq\Upsilon_1-2,\ l\leq{\Upsilon_2}-3,$
$$\Big|\partial^k_{h_5}\partial^l_{t_5} R_{5}\Big|\leq C{h_5}^{-\frac{1}{4}-\frac{k}{2}},$$
     $$\ \Big|\partial^k_{h_5}\partial^l_{t_5}\partial^j_\theta R_{5}\Big|\leq C{h_5}^{-\frac{k}{2}+\frac{j-1}{2}},\ j\geq1.$$
     \end{lemma}
    \begin{proof} Following Lemma \ref{L11} and similar to the proof of Lemma \ref{L08}, the estimates are obtained by a direct computation. \end{proof}

 Without  causing confusion, $\alpha(h_5,t_5)+[R_{41}](h_5,t_5)$ is still denoted by $\alpha(h_5,t_5)$, therefore (\ref{I05-2}) is rewritten as
 \begin{eqnarray}\label{I05-1}
    I_5(h_5,t_5,\theta)&=&\alpha(h_5,t_5)
    +R_{5}(h_5,t_5,\theta)
 \end{eqnarray}
 with
    \begin{eqnarray}\label{alpha02}
     c{h_5}^{\frac{1}{2}-k}\leq \Big|\partial^k_{h_5}\alpha(h_5,t_5)\Big|\leq C{h_5}^{\frac{1}{2}-k},\  k=0,1,2,
     \end{eqnarray}
  \begin{eqnarray}\label{alpha03}
   \Big|\partial^k_{h_5}\partial^l_{t_5}\alpha(h_5,t_5)\Big|\leq C{h_5}^{\frac{1}{2}-k},\  k\leq\Upsilon_1-1,\ l\leq{\Upsilon_2}-2. \end{eqnarray}

   \subsection{The improvement of estimates on derivatives of the oscillating terms}\label{4.3}
In the remain part of this section, we will deal with the oscillating term $R_5$. By intuition, it seems plausible to reduce the perturbation to be small enough by repeating the procedure in Subsection \ref{s4.1}. However, this method does not work because of the poor estimates on derivatives of the perturbation. Fortunately it can improve
the poor estimates with respect to $\theta$. This will help us to obtain a nearly integrable superlinear system in Subsection \ref{4.5}.
\begin{lemma}\label{P5s}
Given $\nu\in\mathbb{Z}^+$, there exists a  transformation
 $\Phi_{6,\nu}:(h_6,\ t_6,\ \theta)\rightarrow(h_5,\ t_5,\ \theta)$, such that
\begin{eqnarray}\label{I06}
   \nonumber
   I_6(h_6,\ t_6,\ \theta)&=&I_5\circ\Phi_{6,\nu}(h_6,\ t_6,\ \theta)\\
  &=& \alpha(h_6,\ t_6)
    +R_{6}(h_6,\ t_6,\ \theta),
 \end{eqnarray}
and for $h_6$ large enough, $\theta,\ t_6\in \mathbb{S}^1$,
$l\leq{\Upsilon_2}-\nu-3$, $ k+j\leq\Upsilon_1-\nu-2$, it holds that

 \begin{eqnarray}\label{R06}
    \Big|\partial^k_{h_6}\partial^l_{t_6}\partial^j_\theta R_{6}\Big|\leq Ch_6^{-\frac{1}{4}-\frac{k}{2}},\ j=0,1,\ldots,\nu.
 \end{eqnarray}
     \end{lemma}
     \begin{remark}\label{}
    Lemma \ref{P5s} shows that with the cost of {reducing the} smoothness on $t$, the smoothness of
the perturbation on $\theta$ and the corresponding estimate can be improved.
 \end{remark}
\begin{proof}
To prove lemma \ref{P5s}, {we give the following iteration lemma firstly}.
\begin{lemma}\label{P53-1}
Assume   Hamiltonian
$$
I=\alpha(h, t)
    +R(h,t,\theta)
$$
with $\alpha$ defined {in (\ref{I05-1})} and $R(h,t,\theta)$ satisfying that
for $h$ large enough, $\theta,\ t\in
\mathbb{S}^1$,\ $l\leq{\Upsilon_2}-i-3$, $ k+j\leq\Upsilon_1-i-2$,
$$\Big|\partial^k_{h}\partial^l_{t}\partial^j_\theta R\Big|\leq Ch^{-\frac{1}{4}-\frac{k}{2}},\ j=0,\ 1,\ \cdots,\ i.$$
Then there exists a  transformation
 $\Phi_{+}:(h_+,\ t_+,\ \theta)\rightarrow(h,\ t,\ \theta)$, such that
\begin{eqnarray}\label{}
   \nonumber
   I_{+}(h_{+},t_{+},\theta)&=&I\circ\Phi_{+}(h_,\ t,\ \theta)\\
    \nonumber &=& \alpha(h_{+},\ t_{+})
    +R_{+}(h_{+},\ t_{+},\ \theta).
 \end{eqnarray}
Moreover for $h_{+}\gg 1$, $\theta,\ t_{+}\in
\mathbb{S}^1$,\ $l\leq{\Upsilon_2}-(i+1)-3$, $ k+j\leq\Upsilon_1-(i+1)-2$, it holds that
$$\Big|\partial^k_{h_+}\partial^l_{t_+}\partial^j_\theta R_+\Big|\leq Ch_+^{-\frac{1}{4}-\frac{k}{2}},\ j=0,\ 1,\ \cdots, \ i+1.$$
     \end{lemma}

  \begin{proof}
 Set
$\Phi_{+}:(h_{+},\ t_{+},\ \theta)\rightarrow(h,\ t,\ \theta)$ implicitly given
by
    \begin{equation}\label{T0{5+}}\left\{ \begin{array}{l}
    \nonumber   h=h_{+}+\partial_{ t}S_{+}(h_{+},t,\theta)\\
     \nonumber   t_{+}\ =t+\partial_{ h_{+}}S_{+}(h_{+},t_,\theta)
     \end{array}\right.
 \end{equation}
 with the generating function $S_{+}(h_{+},t,\theta)$ determined by
\begin{equation}\label{S05+}
   \nonumber S_{+}(h_{+},t,\theta)=-\int_0^\theta(R(h_{+}, t,\theta) -[R](h_{+}, t))d\theta.
\end{equation}
It is easy to show that, for $ k+j\leq \Upsilon_1-(i+1)-2,\
l\leq{\Upsilon_2}-(i+1)-3$,
$$\Big|\partial^k_{h_{+}}\partial^l_{t}\partial^j_\theta S_{+}(h_{+},t,\theta)\Big|\leq Ch_{+}^{-\frac{1}{4}-\frac{k}{2}},\ j=0,\ 1,\ \cdots,\ i+1,$$
 which means that the  smoothness of
$ S_{+}(h_{+},t,\theta)$ depending on $\theta$ is better than
$R_{}.$

Under $\Phi_{+}$, the Hamiltonian $I$ is transformed into
$I_{+}$ as following
 \begin{eqnarray}\label{I05+}
   \nonumber I_{+}(h_{+},t_{+},\theta)&=&\alpha(h_{+}+\partial_{ t}S_{+},t_{+}-\partial_{ h_{+}}S_{+})+R_{}(h_{+}+\partial_{ t_{}}S_{+},t_{},\theta)+\partial_{\theta}S_{+}\\
   \nonumber   &=&   { \alpha(h_{+},t_{+})+R_{{+}}(h_{+},t_{+},\theta),}
 \end{eqnarray}
 where
  \begin{eqnarray}\label{R05+}
    \nonumber  R_{{+}}(h_{+},t_{+},\theta)&=&[R_{}](h_{+}, t_{+})\\
       \nonumber  &&+\int_0^1\partial_I\alpha(h_{+}+\mu\partial_{ t_{}}S_{+},t_{})\partial_{t_{}}S_{+}d\mu-\int_0^1\partial_t\alpha(h_{+},t_{+}-\mu\partial_{ h_{+}}S_{+})\partial_{h_{+}}S_{+}d\mu\\
     \nonumber   &&+\int_0^1\partial_{t_{}}[R_{}](h_{+},t_{+}-\mu\partial_{ h_{+}}S_{+})\partial_{h_{+}}S_{+}d\mu+\int_0^1\partial_{h_{}}R_{}(h_{+}+\mu\partial_{ t_{}}S_{+},t_{},\theta)\partial_{
       t_{}}S_{+}d\mu.
 \end{eqnarray}

    {  Note that the worst term  in $R_+$ is} $\int_0^1\partial_t\beta(h_{+},t_{+}-\mu\partial_{ h_{+}}S_{+})\partial_{h_{+}}S_{+}d\mu$,   therefore the estimates  are calculated directly. \end{proof}

The proof of Lemma \ref{P5s} is completed by using Lemma
\ref{P53-1} $\nu$ times and
$\Phi_{6,\nu}:(h_6,t_6)\rightarrow(h_{},t_{})$ is the composition of
$\nu$ corresponding transformations. \end{proof}

    \subsection{Exchange  the roles of $(h_6,t_6)$ and $(I_6,\theta)$}

   With Lemma \ref{P5s},   {  we have  better estimates about derivatives of the new perturbation with respect to $\theta$ than those for the old one}. Thus the method of exchanging the roles of angle and time will work again.

   Consider the Hamiltonian (\ref{I06}). Assume $h_6=N(\rho,t_6)$ be the inverse function of  $\rho=\alpha(h_6,t_6)$ with respect to the variable $\rho$.
   { Noting that $ {\partial_{h_6} I_6} >ch_6^{-\frac{1}{2}}>0$ as $h_6\rightarrow\infty,$
    for large $h_6$ we can solve (\ref{I06}) for it as the following form:}
    \begin{eqnarray}\label{h6}
  h_6(I_6,\theta,t_6)=N(I_6,t_6)+P(I_6,\theta,t_6).
 \end{eqnarray}

With (\ref{I06}) and (\ref{h6}), we have
\begin{eqnarray}\label{Inu-2}
  \nonumber  I_6&=&\alpha(N+P,t_6) +R_{6}(N+P,t_6,\theta)\\
 \nonumber   &=&\alpha(N,t_6)+\partial_{h_6}\alpha(N,t_6)P+\int_0^1\int_0^1\partial^2_{h_6}\alpha(N+s\mu P,t_6)\mu P^2dsd\mu\\
   \nonumber &&+R_{6}(N+P,t_6,\theta).
 \end{eqnarray}
Note that $I_6=\alpha(N,t_6)$, then
\begin{eqnarray}\label{Inu-3}
  \nonumber  0&=&\partial_{h_6}\alpha(N,t_6)P+\int_0^1\int_0^1\partial^2_{h_6}\alpha(N+s\mu P,t_6)\mu P^2dsd\mu+R_{6}(N+P,t_6,\theta).
 \end{eqnarray}
 Implicitly,
   \begin{eqnarray}\label{P001}
    P&=&-\frac{1}{\partial_{h_6}\alpha(N,t_6)}\{R_{6}(N+P,t_6,\theta)+\int_0^1\int_0^1\partial^2_{h_6}\alpha(N+s\mu P,t_6)\mu
    P^2dsd\mu\}.
    \end{eqnarray}

\vskip 0.3cm
   { In the following, we give} the estimates on $N(I_6,t_6)$ and $P(I_6,\theta,t_6)$.
 \begin{lemma}\label{L14}
 For $I_6$ large enough, $\theta,\ t_6\in \mathbb{S}^1$ and $\nu$ defined as in Lemma \ref{P5s}, it holds that

\begin{eqnarray}\label{L14-1}
cI_{6}^{2-k}\leq\Big|\partial^k_{I_6}
N(I_6,t_6)\Big|\leq CI_{6}^{2-k},\  k=0,1,2;
 \end{eqnarray}
 \begin{eqnarray}\label{L14-3}
\Big|\partial^k_{I_6}\partial^l_{t_6}
N(I_6,t_6)\Big|\leq CI_{6}^{2-k},\  k\leq\Upsilon_1-1,\
l\leq{\Upsilon_2}-2;
 \end{eqnarray}
 and
 \begin{eqnarray}\label{L14-2}
\Big|\partial^k_{I_6}\partial^j_{\theta}\partial^l_{t_6}
P(I_6,\theta,t_6)\Big|\leq CI_{6}^{\frac{1}{2}},\
k+j\leq\Upsilon_1-\nu-2,\ j\leq\nu,\ l\leq{\Upsilon_2}-\nu-3.
 \end{eqnarray}

    \end{lemma}
\begin{proof}
 Although the proof is similar to the one of Lemma
\ref{L05}, the details are different. In Lemma
\ref{L05}, the polynomial growth condition is available, while in this lemma it is not the case.
   { We show} a complete proof as follows.
\vskip0.3cm
 (i) Firstly, we estimate $N(I_6,t_6).$ Note that $\alpha(N(I_6,t_6),t_6)\equiv I_6,$ then
 \begin{eqnarray}\label{}
  \nonumber   cI_6^2\leq\Big|N\Big|\leq CI_6^2,
 \end{eqnarray}
 and
\begin{eqnarray}\label{}
  \nonumber \partial_{h_6}\alpha\cdot\partial_{I_6}N=1,\ \  \partial_{h_6}\alpha\cdot\partial_{t_6}N+ \partial_{t_6}\alpha=0.
 \end{eqnarray}
 Thus from (\ref{alpha02}) and (\ref{alpha03}), it follows that
    \begin{eqnarray}\label{}
  \nonumber  cI_6\leq\Big|\partial_{I_6}N\Big|\leq CI_6,\ \  \Big|\partial_{t_6}N\Big|\leq CI_6^2.
 \end{eqnarray}
Generally, for $2\leq k+j\leq\Upsilon_1-1$ and $l\leq{\Upsilon_2}-2$,
  \begin{eqnarray}\label{}
 \nonumber \partial^k_{I_6}\partial^l_{t_6}\alpha(N(I_6,t_6),t_6)=0.
 \end{eqnarray}

 Using Leibniz's rule, the left hand side of the equation,
$\partial^k_{I_6}\partial^l_{t_6}\alpha(N(I_6,t_6),t_6)$ is the
sum of terms$$ (\partial^u_{h_6}\partial^v_{t_6}\alpha)
\Pi_{i=1}^u\partial^{k_i}_{I_6}\partial^{l_i}_{t_6}N$$ with
$1\leq u+v\leq k+l,\ \sum_{i=1}^{u}k_i =k,v+\sum_{i=1}^{u}l_i
=l$,\ and $k_i+l_i\geq 1,\ i=1,\ldots,u.$ Following (\ref{alpha02}) and (\ref{alpha03}),
(\ref{L14-1}) and (\ref{L14-3}) are obtained inductively.
\vskip0.3cm
(ii) Secondly, from (\ref{P001}) we obtain $\big|P\big|\leq CI_6^{\frac{1}{2}}$
and
 \begin{eqnarray}\label{P001-1}
    -\partial_{h_6}\alpha(N,t_6)\cdot P&=&R_{6}(N+P,t_6,\theta)+\int_0^1\int_0^1\partial^2_{h_6}\alpha(N+s\mu P,t_6)\mu P^2dsd\mu.
 \end{eqnarray}

 Suppose \begin{eqnarray}\label{P001-3}
\Big|\partial^k_{I_6}\partial^j_{\theta}\partial^l_{t_6}
P(I_6,\theta,t_6)\Big|\leq CI_{6}^{\frac{1}{2}}
 \end{eqnarray}
 holds for $k+j+l< m,\ k+j\leq\Upsilon_1-\nu-2,\ l\leq{\Upsilon_2}-\nu-3$.

When $k+j+l=m,\ k+j\leq\Upsilon_1-\nu-2,\ l\leq{\Upsilon_2}-\nu-3,$ consider the left hand
side of (\ref{P001-1}), we claim that
\begin{eqnarray}\label{P001-2}
\Big|\partial^k_{I_6}\partial^l_{t_6}(\partial_{h_6}\alpha(N,t_6))\Big|\leq
CI_6^{-1-k}.
 \end{eqnarray}

 In fact, by Leibniz's rule,
 $\partial^k_{I_6}\partial^l_{t_6}(\partial_{h_6}\alpha(N,t_6))$
 is the sum of terms
$$ (\partial^{u+1}_{h_6}\partial^v_{t_6}\alpha)
\Pi_{i=1}^u\partial^{k_i}_{I_6}\partial^{l_i}_{t_6}N,$$ with
$1\leq u+v\leq k+l,\ \sum_{i=1}^{u}k_i =k,v+\sum_{i=1}^{u}l_i
=l$,\ and $k_i+l_i\geq 1,\ i=1,\ldots,u.$
Then, $$
\Big|\partial^k_{I_6}\partial^l_{t_6}(\partial_{h_6}\alpha(N,t_6))\Big|\leq
CI_6^{-1-2u+2u-k}\leq CI_6^{-1-k}.$$

Thus, differentiating the left hand side of (\ref{P001-1}), we
have
\begin{eqnarray}\label{pp1}
\partial^k_{I_6}\partial^j_{\theta}\partial^l_{t_6}(\partial_{h_6}\alpha(N,t_6)\cdot
P)&=&\partial_{h_6}\alpha(N,t_6)\cdot
\partial^k_{I_6}\partial^j_{\theta}\partial^l_{t_6} P+\widetilde{P}
\end{eqnarray}
where $\widetilde{P}$ is the sum of terms
$$
\partial^{k_1}_{I_6}\partial^{l_1}_{t_6}(\partial_{h_6}\alpha(N,t_6))\cdot
\partial^{k_2}_{I_6}\partial^j_{\theta}\partial^{l_2}_{t_6}P$$
with $k_1+k_2=k,l_1+l_2=l, k_2+l_2+j<m.$ From
(\ref{P001-3}), (\ref{P001-2}), it holds that $|\widetilde{P}|\leq
I^{-\frac{1}{2}}.$

Now, consider the right hand side of (\ref{P001-1}).
\begin{eqnarray}\label{}
   \nonumber \partial^k_{I_6}\partial^j_{\theta}\partial^l_{t_6} {R_6(N+P,t_6,\theta)}
 \end{eqnarray}
 is the sum of
terms$$ \partial^p_{h_6}\partial^q_{\theta}\partial^r_{t_6}
R_6(h_6,t_6,\theta)
\Pi_{i=1}^p\partial^{k_i}_{I_6}\partial^{j_i}_{\theta}\partial^{l_i}_{t_6}(N+P),$$
with $1\leq p+q+r\leq k+j+l,\ \sum_{i=1}^{p}k_i
=k,q+\sum_{i=1}^{p}j_i =j,r+\sum_{i=1}^{p}l_i =l$,\ and
$k_i+j_i+l_i\geq 1,\ i=1,\ldots,p.$



From
 (\ref{R06}), it is easy to show that
$$ \Big|\partial^p_{h_6}\partial^j_{\theta}\partial^r_{t_6} R_6(h_6,t_6,\theta)\Big|\leq Ch_6^{-\frac{1}{4}-\frac{p}{2}}\sim CI_6^{-\frac{1}{2}-p},$$
which, combined with (\ref{L14-3}), implies that
  $$ \Big|\partial^p_{h_6}\partial^j_{\theta}\partial^r_{t_6} {R_6(h_6,t_6,\theta)}
\cdot\Pi_{i=1}^p\partial^{k_i}_{I_6}\partial^{l_i}_{t_6}N\Big|\leq C
I_6^{-\frac{1}{2}-p+2p-k}\leq  CI_6^{-\frac{1}{2}},$$ and

$$ \sum\partial^p_{h_6}\partial^q_{\theta}\partial^r_{t_6}
R_6
\cdot\Pi_{i=1}^p\partial^{k_i}_{I_6}\partial^{j_i}_{\theta}\partial^{l_i}_{t_6}P=\partial_{h_6}R_6\cdot\partial^{k}_{I_6}\partial^{j}_{\theta}\partial^{l}_{t_6}P+\sum_{k_i+j_i+l_i<k+j+l}\partial^p_{h_6}\partial^q_{\theta}\partial^r_{t_6}
R_6
\cdot\Pi_{i=1}^p\partial^{k_i}_{I_6}\partial^{j_i}_{\theta}\partial^{l_i}_{t_6}P.$$
From the assumption, it follows that
$$|\sum_{k_i+j_i+l_i<k+j+l}\partial^p_{h_6}\partial^q_{\theta}\partial^r_{t_6}
R_6
\cdot\Pi_{i=1}^p\partial^{k_i}_{I_6}\partial^{j_i}_{\theta}\partial^{l_i}_{t_6}P|\leq
CI_6^{-\frac{1}{2}-p+\frac{p}{2}}\leq CI_6^{-\frac{1}{2}}.$$ Hence
\begin{eqnarray}\label{pp2}
 \partial^k_{I_6}\partial^j_{\theta}\partial^l_{t_6}
 {R_6(N+P,t_6,\theta)}=\partial_{h_6}R_6\cdot\partial^{k}_{I_6}\partial^{j}_{\theta}\partial^{l}_{t_6}P+\hat{P}
 \end{eqnarray}
with $|\hat{P}|\leq CI_6^{-\frac{1}{2}}.$

Finally, consider
\begin{eqnarray}\label{}
\nonumber
\partial^k_{I_6}\partial^j_{\theta}\partial^l_{t_6}(\partial^2_{h_6}\alpha(N+s\mu P,t_6)\mu
P^2).
\end{eqnarray}
By the same method, we have
\begin{eqnarray}\label{pp3}
\nonumber \partial^2_{h_6}\alpha(N+s\mu P,t_6)\mu
P^2=(\partial^3_{h_6}\alpha(N+s\mu P,t_6)\mu
P^2+2\partial^2_{h_6}\alpha(N+s\mu P,t_6)\mu
P)\cdot\partial^k_{I_6}\partial^j_{\theta}\partial^l_{t_6}P+\breve{P}\\
\end{eqnarray}
with $|\breve{P}|\leq CI_6^{-2}.$

 With
(\ref{pp1}), (\ref{pp2}) and (\ref{pp3}), by induction,
 we get (\ref{L14-2}).
  \end{proof}



 \subsection{A nearly integrable system}\label{4.5}

 For convenience, we redefine the variables as
 $(I_6,\theta,t_6,h_6)\rightarrow(\rho,\theta,\tau,h)$, and
 (\ref{h6}) is rewritten by
   \begin{eqnarray}\label{rho}
  h(\rho,\theta,\tau)=N(\rho,\tau)+P(\rho,\theta,\tau).
 \end{eqnarray}
Inductively, consider the Hamiltonian
 \begin{eqnarray}\label{}
   \nonumber h_s(\rho_{ {s}},\theta_{ {s}},\tau)=N(\rho_{ {s}},\tau)+M_{{ {s}}}(\rho_{ {s}},\tau)+P_{{ {s}}}(\rho_{ {s}},\theta_{ {s}},\tau),\
  \ { {s}}=0,1,2,\cdots
 \end{eqnarray}
 with
 \begin{eqnarray}\label{}
   \nonumber (\rho_0,\theta_0)=(\rho,\theta),\ \ \ \ M_0=0,\ \ \ \ P_0=P;
 \end{eqnarray}
 and $M_{ {s}}(\rho_{ {s}},\tau),\ P_{{ {s}}}(\rho_{ {s}},\theta_{ {s}},\tau)$ satisfying
 \begin{eqnarray}\label{ite01}
 \Big|\partial^k_{\rho_{ {s}}}\partial^l_{\tau} M_{ {s}}(\rho_{ {s}},\tau)\Big|\leq
 C\rho_{ {s}}^{\frac{1}{2}}, \ \ \ k+j\leq{\Upsilon_1}-\nu-{ {s}}-1,\
 l\leq{\Upsilon_2}-\nu-{ {s}}-2;
 \end{eqnarray}
  \begin{eqnarray}\label{ite02}
\nonumber\Big|\partial^k_{\rho_{ {s}}}\partial^j_{\theta_{ {s}}}\partial^l_{\tau}
P_{ {s}}(\rho_{ {s}},\theta_{ {s}},\tau)\Big|\leq C\rho_{ {s}}^{-\frac{1}{2}},\ \ \ j\leq\nu,\
k+j\leq{\Upsilon_1}-\nu-{ {s}}-2,\ l\leq{\Upsilon_2}-\nu-{ {s}}-3.\\
 \end{eqnarray}
for $\rho_{ {s}}$ large enough. Thus we have
\begin{lemma}\label{L15}
 Suppose the Hamiltonian $h_{ {s}}$  with
 ${ {s}}=\kappa-1$ satisfies (\ref{ite01}), (\ref{ite02}).
 Then there exists a canonical transformation
 $\Psi_\kappa:(\rho_\kappa,\theta_\kappa, \tau)\rightarrow(\rho_{\kappa-1},\theta_{\kappa-1},\tau)$
 such that
  the new Hamiltonian $h_{ {s}}$  with ${ {s}}=\kappa$ satisfies (\ref{ite01}), (\ref{ite02}).
    \end{lemma}
    \begin{proof}
    Suppose $h_{ {s}}$ with (\ref{ite01}), (\ref{ite02}) holds for ${ {s}}=\kappa-1$ (case
    $\kappa=1$ is already satisfied by (\ref{L14-2})). Set
    $\Psi_\kappa:(\rho_\kappa,\theta_\kappa, \tau)\rightarrow(\rho_{\kappa-1},\theta_{\kappa-1}, \tau)$
    being defined implicitly by
    \begin{equation}\label{Tnu01}\left\{ \begin{array}{l}
    \rho_{\kappa-1}=\rho_\kappa+\partial_{
    {\theta_{\kappa-1}}}Q_\kappa(\rho_\kappa,{\theta_{\kappa-1}},\tau)\\
      \theta_\kappa ={\theta_{\kappa-1}}+\partial_{
      \rho_\kappa}Q_\kappa(\rho_\kappa,{\theta_{\kappa-1}},\tau)
     \end{array}\right.
 \end{equation}
 with the generating function $Q_\kappa(\rho_\kappa,{\theta_{\kappa-1}},\tau)$ determined
 by
\begin{equation}\label{Q01}
 Q_\kappa(\rho_\kappa,{\theta_{\kappa-1}},\tau)=-\int_0^{\theta_{\kappa-1}}\frac{1}{\partial_{\rho_{\kappa-1}}N(\rho_\kappa,\tau)}(P_{\kappa-1}(\rho_\kappa,
 {\theta_{\kappa-1}},\tau) -[P_{\kappa-1}](\rho_\kappa, \tau))d{\theta_{\kappa-1}}.
\end{equation}
 {Under $\Psi_{\kappa}$, the Hamiltonian $h_{\kappa-1}$ is transformed into $h_{\kappa}$
as following
 \begin{eqnarray}\label{hnew02-1}
   \nonumber h_{\kappa}(\rho_\kappa,\theta_\kappa,\tau)&=&N(\rho_\kappa+\partial_{
   {\theta_{\kappa-1}}}Q_\kappa,\tau)+M_{\kappa-1}(\rho_\kappa+\partial_{
   {\theta_{\kappa-1}}}Q_\kappa,\tau)\\
     \nonumber&&+P_{\kappa-1}(\rho_\kappa+\partial_{
     {\theta_{\kappa-1}}}Q_\kappa,{\theta_{\kappa-1}},\tau)+\partial_{\tau}Q_\kappa\\
                \nonumber    &=&N(\rho_\kappa,\tau)+M_\kappa(\rho_\kappa,\tau)+P_\kappa(\rho_\kappa,\theta_\kappa,\tau),
    \end{eqnarray}
where
$M_\kappa(\rho_\kappa,\tau)=M_{\kappa-1}(\rho_\kappa,\tau)+[P_{\kappa-1}](\rho_\kappa,\tau)$
and
  \begin{eqnarray}\label{P01}
   \nonumber
   P_\kappa(\rho_\kappa,\theta_\kappa,\tau)&=&\int_0^1\partial_{\rho_{\kappa-1}}\{M_{\kappa-1}+P_{\kappa-1}\}(\rho_\kappa+\mu\partial_{
   {\theta_{\kappa-1}}}Q_\kappa,\tau)\partial_{{\theta_{\kappa-1}}}Q_\kappa(\rho_\kappa,{\theta_{\kappa-1}},\tau)d\mu\\
    \nonumber    &&+\int_0^1\int_0^1\partial^2_{\rho_{\kappa-1}}N(\rho_\kappa+s\mu\partial_{
      {\theta_{\kappa-1}}}Q_\kappa,\tau)\mu(\partial_{
      {\theta_{\kappa-1}}}Q_\kappa)^2dsd\mu+\partial_{\tau}Q_\kappa.
 \end{eqnarray}}

 From (\ref{L14-1}), (\ref{ite02}), (\ref{Q01}), it follows that
  $$\Big|\partial^k_{\rho_{\kappa}}\partial^j_{{\theta}_{\kappa-1}}\partial^l_{\tau}
  Q_\kappa(\rho_{\kappa},{\theta}_{\kappa-1},\tau)\Big|\leq
  C\rho_{\kappa}^{-\frac{\kappa}{2}},\quad j\leq\nu,\ k+j\leq{\Upsilon_1}-\nu-\kappa-1,\
  l\leq{\Upsilon_2}-\nu-\kappa-2$$
 for $\rho_{\kappa}$ large enough.

 Finally, the estimates on $M_\kappa,\ P_\kappa$ are similar to those in the proof of lemma
 \ref{L08}. We omit it.\end{proof}


Now, under a series of transformations as  $\Psi_1$, $\Psi_2$, $\ldots,\
\Psi_\kappa$, $h$ is transformed into
$h_{\kappa}$ with
 \begin{eqnarray}\label{hnu0k}
   \nonumber   h_{\kappa}(\rho_\kappa,\theta_\kappa,\tau)=N(\rho_\kappa,\tau)+M_\kappa(\rho_\kappa,\tau)+P_\kappa(\rho_\kappa,\theta_\kappa,\tau)
 \end{eqnarray}
 satisfying (\ref{ite01}) and (\ref{ite02}) with ${ {s}}=\kappa.$

From (\ref{ite02}), we find that
\begin{equation}\label{par01}
  {\Upsilon_2}-\nu-\kappa\geq3.
 \end{equation}

     \section{Existence of invariant curves}\label{s5}

Consider  the system
  with Hamiltonian $h_{\kappa}(\rho_\kappa,\theta_\kappa,\tau)$, that is
 \begin{equation}\label{P02}\left\{ \begin{array}{l}
   \displaystyle  \frac{d \theta_\kappa}{d\tau}=\partial_{ \rho_\kappa}(N+M_{\kappa})+\partial_{ \rho_\kappa} P_{\kappa},\\
   \\
    \displaystyle    \frac{d \rho_\kappa}{d\tau}=-\partial_{\theta_\kappa} P_{\kappa}.
     \end{array}\right.
      \end{equation}
      The Poincar\'e map $\mathcal{P}$ of (\ref{P02}) is of the form
 \begin{equation}\label{P03}\left\{ \begin{array}{l}
 \theta_{\kappa+}=\theta_{\kappa}+\gamma(\rho_\kappa)+F_1(\rho_\kappa,\theta_{\kappa}),\\
   \\
    \displaystyle    \rho_{\kappa+}=\rho_\kappa+F_2(\rho_\kappa,\theta_{\kappa}).
     \end{array}\right.
      \end{equation}
where
$(\rho_\kappa,\theta_\kappa)=(\rho_\kappa(0),\theta_\kappa(0))$,
and
\begin{equation}\label{} \begin{array}{l}
   \nonumber\displaystyle\gamma(\rho_\kappa)= \int_0^{2\pi}\partial_{ \rho_\kappa}(N+M_{\kappa})(\rho_\kappa,\tau)d\tau;  \\
   \\
 \displaystyle   F_1(\rho_\kappa,\theta_{\kappa})=\int_0^{2\pi}\partial_{ \rho_\kappa} P_{\kappa}(\rho_\kappa(\tau)),\theta_\kappa(\tau),\tau)d\tau  \\
   \displaystyle\ \ \ \ \ \ \ \ \ \ \ \ \ \   +\int_0^{2\pi}\partial_{ \rho_\kappa}(N+M_{\kappa})(\rho_\kappa(\tau),\tau)d\tau- \int_0^{2\pi}\partial_{ \rho_\kappa}(N+M_{\kappa})(\rho_\kappa,\tau)d\tau;\\
   \\
\displaystyle
F_2(\rho_\kappa,\theta_{\kappa})=-\int_0^{2\pi}\partial_{\theta_\kappa}
P_{\kappa}(\rho_\kappa(\tau)),\theta_\kappa(\tau),\tau)d\tau.\\
\\

     \end{array}
          \end{equation}
  Denote
$\gamma(\rho_\kappa,\tau)=\int_0^{\tau}\partial_{
\rho_\kappa}(N+M_{\kappa})(\rho_\kappa,s)ds, $ and
$\gamma(\rho_\kappa)=\gamma(\rho_\kappa,1)$.      By direct computations, we have the estimates on the map $\mathcal{P}$ as following.
 \begin{lemma}\label{L16}
Given $\rho_\kappa$ large enough and ${\theta_\kappa}\in \mathbb{S}^1$,  it
holds that
$$c\rho_\kappa\leq \gamma(\rho_\kappa)\leq C\rho_\kappa,\ \ \ \Big| \gamma^{(k)}(\rho_\kappa)\Big|\leq C\rho_\kappa^{\frac{1}{2}},\quad k\leq{\Upsilon_1}-\nu-\kappa-2;$$
  and for $k+j\leq{\Upsilon_1}-\nu-\kappa-3$, $ j\leq\nu-1,$
  $$\Big| F_1(\rho_\kappa,{\theta_\kappa})\Big|\leq C\rho_\kappa^{\frac{1}{2}-\frac{\kappa}{2}},\ \Big| F_2(\rho_\kappa,{\theta_\kappa})\Big|\leq C\rho_\kappa^{\frac{1}{2}-\frac{\kappa}{2}};$$
   $$\Big|\partial^k_{\rho_\kappa}\partial^j_{{\theta_\kappa}} F_1(\rho_\kappa,{\theta_\kappa})\Big|\leq C\rho_\kappa^{\frac{1}{2}-\frac{\kappa}{2}},\ \Big|\partial^k_{\rho_\kappa}\partial^j_{{\theta_\kappa}} F_2(\rho_\kappa,{\theta_\kappa})\Big|\leq C\rho_\kappa^{\frac{1}{2}-\frac{\kappa}{2}}.$$
    \end{lemma}

\vskip 0.4cm

    Moreover, the following twist condition holds true.
 \begin{lemma}\label{L17}
Given $A_1$ large enough, there exists an interval $[A_1,\
A_1+A_1^{-\frac{2}{3}}]$ such that
 for $\rho_{\kappa}\in[A_1,\
A_1+A_1^{-\frac{2}{3}}]$ and $\theta_\kappa\in \mathbb{S}^1$ it holds that
 $$c\leq\Big| \gamma'(\rho_\kappa)\Big|\leq C.$$
    \end{lemma}
\begin{proof}
   Denote
   $\gamma(\rho_\kappa)=\gamma_1(\rho_\kappa)+\gamma_2(\rho_\kappa)$
   with
$$\gamma_1(\rho_\kappa)=\int_0^{2\pi}\partial_{
\rho_\kappa}N(\rho_\kappa,s)ds, \ \ \
\gamma_2(\rho_\kappa)=\int_0^{2\pi}\partial_{
\rho_\kappa}M_{\kappa}(\rho_\kappa,s)ds. $$ It is easy to verify
that $$c\leq\Big| \gamma_1'(\rho_\kappa)\Big|\leq C,$$
and
$$\Big|
\gamma_2^{(k)}(\rho_\kappa)\Big|\leq C\rho_\kappa^{\frac{1}{2}},\ k\leq{\Upsilon_1}-\nu-\kappa-2.$$
Note that, for $A$ large enough,
\begin{eqnarray}\label{}
 \nonumber \gamma(\rho_\kappa)\Big|_{A}^{2A}& =&
\int_{A}^{2A}\gamma_1'(\rho_\kappa) d\rho_\kappa+\gamma_2(\rho_\kappa)\Big|_{A}^{2A}\\
 \nonumber\\
 \nonumber&\geq&
cA-C A^{\frac{1}{2}}\\
 \nonumber\\
 \nonumber &\geq&
c_1A.
\end{eqnarray}
By Mean Value Theorem of the integral, there exists a point
$\xi\in(A,2A)$, such that $\gamma'(\xi)\geq c>0.$ Note that
$|\gamma''(\rho_\kappa)|\leq C\rho_\kappa^{\frac{1}{2}}.$
Therefore, we can choose $A_1$ such that $\xi\in[A_1,\
A_1+A_1^{-\frac{2}{3}}]\subset[A,2A]$. This ends the proof of this lemma.  \end{proof}

 Finally, for $A_1\gg1, $ let
$\gamma(\rho_\kappa)=\gamma(A_1)+A_1^{-1}\lambda,\
\lambda\in[1,2]$. Denote $\rho_\kappa=\rho_\kappa(\lambda),\
\lambda\in[1,2],$ then $\rho_\kappa(\lambda)\subset[A_1,\
A_1+A_1^{-\frac{2}{3}}],\ \lambda\in[1,2],$ and
 \begin{equation}\label{lambda}
|\rho_\kappa^{(k)}(\lambda)|\leq CA_1^{-\frac{k+1}{2}},\ k\leq{\Upsilon_1}-\nu-\kappa-2.
      \end{equation}

      In fact, for $k=1$, $\gamma'(\rho_\kappa)\rho'_\kappa(\lambda)=A_1^{-1}$.    Thus from Lemma \ref{L17}, we have $\Big|\rho_\kappa'(\lambda)\Big|\leq CA_1^{-1}.$

   For $k=2$,  $\gamma''(\rho_\kappa)(\rho'_\kappa(\lambda))^2+\gamma'(\rho_\kappa)\rho''_\kappa(\lambda)=0$, then from Lemma \ref{L17} we have $\Big|\rho_\kappa''(\lambda)\Big|\leq CA_1^{-\frac{3}{2}}.$

 Similarly, (\ref{lambda}) is obtained inductively.

Denote $\phi=\theta_{\kappa}$, then the map (\ref{P03}) is changed
into
 \begin{equation}\label{P04}\left\{ \begin{array}{l}
 \phi_{+}=\phi+\gamma(A_1)+A_1^{-1}\lambda+\widetilde{F}_1(\lambda,\phi)\\
   \\
    \displaystyle    \lambda_+=\lambda+\widetilde{F}_2(\lambda,\phi)
     \end{array}\right.
      \end{equation}
      with
      $$\widetilde{F}_1(\lambda,\phi)={F}_1(\rho_\kappa(\lambda),\phi),$$
      $$\widetilde{F}_2(\lambda,\phi)=A_1\Big(\gamma(\rho_\kappa(\lambda)+F_2(\rho_\kappa(\lambda),\theta_{\kappa}))-\gamma(\rho_\kappa(\lambda)\Big).$$
For large $A_1$, by a direct computation, from (\ref{lambda}) and Lemma \ref{L16} we have
$$\Big| \widetilde{F}_1(\lambda,\phi)\Big|\leq CA_1^{\frac{1}{2}-\frac{\kappa}{2}},$$

$$\Big| \widetilde{F}_2(\lambda,\phi)\Big|\leq CA_1^{\frac{3}{2}-\frac{\kappa}{2}},$$
and  for $k+j\leq{\Upsilon_1}-\nu-\kappa-3,\ j\leq\nu-1,$
   $$\Big|\partial^k_{\lambda}\partial^j_{\phi} \widetilde{F}_1(\lambda,\phi)\Big|\leq CA_1^{\frac{1}{2}-\frac{\kappa}{2}-\frac{k+1}{2}},\quad \Big|\partial^k_{\lambda}\partial^j_{\phi} \widetilde{F}_2(\lambda,\phi)\Big|\leq CA_1^{\frac{3}{2}-\frac{\kappa}{2}-\frac{k}{2}}.$$

\noindent \textbf{Proof of Theorem \ref{Th1}.}

 From (\ref{par01}), we have that for $\lambda\in[1,2]$, the map
(\ref{P04}) is close to a small twist map in $C^4$ topology
provided that $\Upsilon_1\geq18$ and
 ${\Upsilon_2}\geq14.$ Moreover, it has the
intersection property, thus the assumptions of Moser's Small Twist
Theorem \cite{Mos1,Rus} are met.
More precisely, given any number $\lambda\in[1,2]$ satisfying
$$\Big|\gamma(A_1)+A_1^{-1}\lambda-\frac{p}{q}\Big|>c|q|^{-5/2}$$
for all integers $p$ and $q\neq0$, there exists a $\mu(\varphi)\in
C^3(\mathbb{R}\setminus2\pi\mathbb{Z})$ such that the curve
$\Gamma=\{(\varphi,\mu(\varphi))\}$ is invariant under the mapping
(\ref{P04}). The image point of a point on $\Gamma$ is obtained by
replacing $\varphi$ by $\varphi+\gamma(A_1)+A_1^{-1}\lambda$.
Hence the system with Hamiltonian $h_6$ has
invariant curve with frequency $\gamma(A_1)+A_1^{-1}\lambda$. Then
the system with Hamiltonian $I_6$ has invariant
curve with frequency $(\gamma(A_1)+A_1^{-1}\lambda)^{-1}$, which implies the systems $I_1$ has invariant curve with frequency
$1+(\gamma(A_1)+A_1^{-1}\lambda)^{-1}$. Consequently the original
system $H$ has invariant curve with frequency
$\omega=\frac{\gamma(A_1)+A_1^{-1}\lambda}{1+\gamma(A_1)+A_1^{-1}\lambda}$.
Note that $I\rightarrow\infty\ as\ A_1\rightarrow\infty.$  It
means that we have found arbitrarily large amplitude invariant
tori in $(x,y,t\ \mbox{mod} 1)$ space, which implies the
boundedness of all the solutions.
Thus
 the proof of Theorem \ref{Th1} is finished.

\section{{On the critical situation}}\label{s6}
 {In this section, we give the  proof of Theorem \ref{Th2} on the critical situation when (\ref{LL0}) holds.}
 {We divide the whole proof into the ``bounded" part and the ``unbounded" part as follows.}
\subsection{ {Bounded results for $0<d<1$}}
 {For the reason that many estimates in this situation are the same as those of Theorem \ref{Th1},
 we will omit the proof of them and pay our attention to the difference between the proofs of   two theorems.}

\vskip 0.4cm
{\it Step 1.}\quad\emph{Action-angle coordinates}
\vskip 0.4cm
Equation (\ref{semi}) is equivalent to   Hamiltonian system
$$
\dot{x}=ny,\qquad \dot{y}=-nx-\frac{1}{n}g(x)+\frac 1n p(t)
$$
with   Hamiltonian
\begin{equation}\label{H0}
  \nonumber H_0=\frac{n^2}{2}(x^2+y^2)+\frac 1n G(x)-\frac 1n x p(t).
\end{equation}
Under action-angle coordinates transformation
$$
x=\sqrt{\frac{2}{n}}I^{\frac 12}\cos{n\theta},\qquad
y=\sqrt{\frac{2}{n}}I^{\frac 12}\sin{n\theta},
$$
$H_0$ is transformed into
\begin{equation}\label{H1}
  \nonumber H_1(I,\theta, t)=I+\frac 1n G(\sqrt{\frac{2}{n}}I^{\frac
12}\cos{n\theta})-\frac 1n \sqrt{\frac{2}{n}}I^{\frac
12}\cos{n\theta} p(t).
\end{equation}
\vskip 0.4cm
{\it Step 2.}\quad\emph{A sublinear system}
\vskip 0.4cm
Now we exchange the roles of angle and time variables. From the argument in Section 2, we have that the Hamiltonian system with Hamiltonian $H_1$ is equivalent to the one with the following Hamiltonian
\begin{equation}\label{I_0}
  \nonumber I_0=h-f_1(h,\theta)-f_2(h,\theta, t)+R_0(h,\theta, t),
\end{equation}
where
$$
f_1(h,\theta)=\frac 1n G(\sqrt{\frac{2}{n}}I^{\frac
12}\cos{n\theta}),\qquad f_2(h, \theta, t)=-\frac 1n
\sqrt{\frac{2}{n}}I^{\frac 12}\cos{n\theta} p(t).
$$
Moreover, $f_1$ satisfies the estimates in Lemma \ref{L01} and the following holds true for $R_{0}$:
\begin{equation}\label{R_0}
  \nonumber |\partial_h^k\partial_t^l\partial_{\theta}^j R_0|\le C\cdot
h^{-k+\frac j2},\quad k+j\le \Upsilon_1+1,\ l\le \Upsilon_2.
\end{equation}

Then with a rotation transformation
$$
 h=h_1,\qquad
      t=t_1+\theta,
      $$
      the Hamiltonian system with the Hamiltonian $I_0$ is equivalent to the one with     Hamiltonian
 \begin{equation}\label{I_1}
   \nonumber I_1=-f_1(h_1, \theta)-f_2(h, \theta+t_1, \theta)+R_1(h_1, t_1, \theta),
 \end{equation}
 where $R_1$ satisfies
 \begin{equation}\label{R_1}
  \nonumber |\partial_{h_1}^k\partial_{t_1}^l\partial_{\theta}^jR_1|\le C\cdot h_1^{-k+\frac j2},\qquad k+j\le \Upsilon_1+1,\ l\le \Upsilon_2.
 \end{equation}
 \vskip 0.4cm
 {\it Step 3.}\quad    \emph{Some canonical transformations}
\vskip 0.4cm
By the transformation $\Phi_2$ defined as in (\ref{S02}), we obtain a new Hamiltonian as follows:
\begin{equation}\label{I_2}
  \nonumber I_2(h_2, t_2, \theta)=-[f_1](h_2)-f_2(h_2, \theta+t_2, \theta)+R_2(h_2, t_2, \theta)
\end{equation}
with the following estimates
\begin{equation}\label{R_2}
  \nonumber |\partial_{h_2}^k\partial_{t_2}^l\partial_{\theta}^jR_2|\le C\cdot
h_2^{-k+\frac j2},\quad k+j\le \Upsilon_1,\ l\le \Upsilon_2-1.
\end{equation}
Under the transformation $\Phi_3$ defined as in (\ref{T03}), $I_2$ is changed into
\begin{equation}\label{I_3}
  \nonumber I_3(h_3, t_3, \theta)=-[f_1](h_3)-[f_2](h_3, t_3)+R_3(h_3, t_3, \theta),
\end{equation}
where $[f_2]$ satisfies the estimates in Lemma \ref{LLc-1} and $R_3$ satisfies
\begin{equation}\label{R_3}
  \nonumber |\partial_{h_3}^k\partial_{t_3}^l\partial_{\theta}^jR_3|\le C\cdot
h_3^{-k+\frac j2},\qquad k+j\le \Upsilon_1,\ l\le \Upsilon_2-1.
\end{equation}

Denote
$\beta(h_3)=[f_1](h_3)-\frac{\sqrt{2}}{\pi}n^{-\frac{3}{2}}
h_3^{\frac 12}\cdot\big|g(+\infty)-g(-\infty)\big|$ and
$$
a(t_3)=\frac{\sqrt{2}}{2\pi}n^{-\frac{3}{2}}\Big(2\big|g(+\infty)-g(-\infty)\big|-\big(\cos
nt_3 \int_0^{2\pi}p({ {s}})
 \cos   n{ {s}} d{ {s}}+\sin nt_3 \int_0^{2\pi}p({ {s}})\sin n{ {s}} d{ {s}}\big)\Big).
 $$
Then it holds that
\begin{equation}\label{f1+f2}
  \nonumber -\alpha(h_3,t_3):=[f_1]+[f_2]=a(t_3)\cdot h_3^{\frac
12}+\beta(h_3).
\end{equation}
Denote $ {A:=\Big |\int_0^{2\pi}p(t)e^{int}dt \Big
  |=2\big|g(+\infty)-g(-\infty)\big|,}$ from (\ref{LL0}), we have that
$$a(t_3)=\frac{\sqrt{2}}{2\pi}n^{-\frac{3}{2}}A(1-\cos(nt_3 +\xi))
$$
 with
   $\tan \xi=\frac{\int_0^{2\pi}p({ {s}})\sin( n{ {s}})d{ {s}}}{\int_0^{2\pi}p({ {s}})\cos (n{ {s}})d{ {s}}}.$
Obviously, $a(t_3)\ge 0$ and $a(t_3)=0$ if and only if the
following holds true:
\begin{equation}\label{t^*}
  \nonumber (\cos(nt_3),\ \sin(nt_3))=\frac{(\int_0^{2\pi}p({ {s}})\cos(n{ {s}})d{ {s}},\ \ \int_0^{2\pi}p({ {s}})\sin(n{ {s}})d{ {s}})}{ \Bigg |\int_0^{2\pi}p(t)e^{int}dt \Bigg
  |}.
  \end{equation}
  Moreover, from (\ref{G+-}) we obtain that
  \begin{lemma}For $h_3$ large enough, $\beta(h_3)$ satisfies
  \begin{equation}\label{betah12}
   \nonumber |\beta^{(k)}(h_3)|\ge  c\cdot h_3^{\frac{1-d}{2}-k}, \qquad k=0,\ 1,\ 2
  \end{equation}
  and
  \begin{equation}\label{betah}
   \nonumber |\beta^{(k)}(h_3)|\le C\cdot h_3^{\frac{1-d}{2}-k},\qquad k\le
  \Upsilon_1+1.
  \end{equation}
\end{lemma}
\begin{proof}Without loss of generality, assume $g(+\infty)\ge g(-\infty)$.
Note that
  \begin{equation}\label{}
  \nonumber \beta(h_3)=\frac{1}{2n\pi}\int_0^{2\pi}G(\sqrt{\frac{2}{n}}h_3^{\frac{1}{2}}\cos{n\theta})d\theta-\frac{\sqrt{2}}{\pi}n^{-\frac{3}{2}}\big(g(+\infty)-g(-\infty)\big)\cdot
h_3^{\frac 12}.
  \end{equation}
 {Then we have}
   \begin{eqnarray}\label{}
\nonumber\beta'(h_3)&=&\frac{\sqrt{2}}{4\pi}n^{-\frac{3}{2}}h_3^{-\frac{1}{2}}\Big(\int_0^{2\pi}g(\sqrt{\frac{2}{n}}h_3^{\frac{1}{2}}\cos{n\theta})\cos{n\theta}d\theta-2\big(g(+\infty)-g(-\infty)\big)
\Big)\\
\nonumber&=&\frac{\sqrt{2}}{4\pi}n^{-\frac{3}{2}}h_3^{-\frac{1}{2}}\sum_{k=1}^n\Big(J_{k+}(h_3) {-}J_{k-}(h_3)\Big)
  \end{eqnarray}
  where
     \begin{eqnarray}\label{}
\nonumber J_{k+}(h_3)&=&\int_{\frac{2k\pi}{n}-\frac{\pi}{2n}}^{\frac{2k\pi}{n}+\frac{\pi}{2n}}\Big(g(\sqrt{\frac{2}{n}}h_3^{\frac{1}{2}}\cos{n\theta})- g(+\infty)\Big)\cos{n\theta}d\theta,\\
\nonumber J_{k-}(h_3)&=&\int_{\frac{2k\pi}{n}+\frac{\pi}{2n}}^{\frac{2k\pi}{n}+\frac{3\pi}{2n}}\Big(g(\sqrt{\frac{2}{n}}h_3^{\frac{1}{2}}\cos{n\theta})-g(-\infty)\Big)\cos{n\theta}d\theta.  \end{eqnarray}
  From (\ref{G+-}),
   \begin{eqnarray}\label{}
\nonumber (\frac{2}{n}h_3)^{\frac{d}{2}}J_{k+}(h_3)&=&\int_{\frac{2k\pi}{n}-\frac{\pi}{2n}}^{\frac{2k\pi}{n}+\frac{\pi}{2n}}\Big(g(\sqrt{\frac{2}{n}}h_3^{\frac{1}{2}}\cos{n\theta})- g(+\infty)\Big)(\sqrt{\frac{2}{n}}h_3^{\frac{1}{2}}\cos{n\theta})^{d}\cos^{1-d}{n\theta}d\theta\\
\\
\nonumber &\rightarrow&s(d)c_+,\ \mbox{as}\ h_3\rightarrow\infty.\end{eqnarray}
with $s(d)$ some positive constant.
Similarly, we have
   \begin{eqnarray}\label{}
\nonumber (\frac{2}{n}h_3)^{\frac{d}{2}}J_{k-}(h_3)&=&\int_{\frac{2k\pi}{n}+\frac{\pi}{2n}}^{\frac{2k\pi}{n}+\frac{3\pi}{2n}}\Big(g(\sqrt{\frac{2}{n}}h_3^{\frac{1}{2}}\cos{n\theta})- g(-\infty)\Big)(\sqrt{\frac{2}{n}}h_3^{\frac{1}{2}}\cos{n\theta})^{d}\cos^{1-d}{n\theta}d\theta\\
\\
\nonumber &\rightarrow&-s(d)c_-,\ \mbox{as}\ h_3\rightarrow\infty.\end{eqnarray}
Thus
\begin{eqnarray}\label{bbeta1}
 (\frac{2}{n}h_3)^{\frac{1+d}{2}}\beta'(h_3)&\rightarrow&\frac{1}{2n\pi}s(d)(c_+ {+}c_-),\ \mbox{as}\ h_3\rightarrow\infty,\end{eqnarray}
which means that
\begin{equation}\label{}
   \nonumber c\cdot h_3^{\frac{1-d}{2}-1}\le |\beta'(h_3)|\le C\cdot h_3^{\frac{1-d}{2}-1}.
  \end{equation}
  With L'Hospital's rule, (\ref{bbeta1}) implies
  \begin{equation}\label{}
   \nonumber c\cdot h_3^{\frac{1-d}{2}}\le |\beta(h_3)|\le C\cdot h_3^{\frac{1-d}{2}}.
  \end{equation}

  Finally, with  (\ref{G+-}) and the method above, the rest of estimates on $\beta^{(k)}(h_3)$ is obtained.
\end{proof}
  Consequently, $I_3$ is rewritten by
\begin{equation}\label{I_3-1}
  \nonumber I_3(h_3, t_3, \theta)=\alpha(h_3, t_3)+R_3(h_3, t_3, \theta),
\end{equation}
with {a weaker twist condition compared with (\ref{alpha02})}:
 \begin{equation}\label{alphah12}
 |\partial^k_{h_3}\alpha(h_3,t_3)|\ge  c\cdot h_3^{\frac{1-d}{2}-k}, \qquad k=0,\ 1,\
  2,
  \end{equation}
  and for $  k\le
  \Upsilon_1+1,\ l\le
  \Upsilon_2,$
  \begin{equation}\label{alphah}
  |\partial^k_{h_3}\partial^l_{t_3}\alpha(h_3,t_3)|\le C\cdot
  h_3^{\frac{1}{2}-k}.
  \end{equation}

  \vskip 0.4cm
  {\it Step 4.}\quad \emph{A nearly integrable system}
\vskip 0.4cm
  Similar to Subsection \ref{4.3}, we have an iteration lemma as follows.
  \begin{lemma}\label{L6.2}
Assume  Hamiltonian
$$
I=\alpha(h, t)
    +R(h,t,\theta)
$$
with $\alpha$ satisfying (\ref{alphah12}), (\ref{alphah}) for $
k\le
  m,\ l\le
  n$,  and
$R(h,t,\theta)$ satisfying
$$\Big|\partial^k_{h}\partial^l_{t}\partial^j_\theta R\Big|\leq Ch^{-\frac{i}{2}-k},\ j=0,\ 1,\ \cdots,\ i$$
 for $h$ large enough,\  $ k+j\leq
m_1$,\ $l\leq n_1$$(m_1\le m,\ n_1\le n)$.

Then there exists a  transformation
 $\Phi_{+}:(h_+,\ t_+,\ \theta)\rightarrow(h,\ t,\ \theta)$, such that
\begin{eqnarray}\label{}
   \nonumber
   I_{+}(h_{+},t_{+},\theta)&=&I\circ\Phi_{+}(h_,\ t,\ \theta)\\
    \nonumber &=& \alpha_+(h_{+},\ t_{+})
    +R_{+}(h_{+},\ t_{+},\ \theta),
 \end{eqnarray}
 with $\alpha_+(h_{+},\ t_{+})=\alpha(h_{+},\ t_{+})+[R](h_{+},\
 t_{+})$ satisfying (\ref{alphah12}) and (\ref{alphah}) for $
k\le
  m_1,\ l\le
  n_1$.
Moreover for $h_{+}\gg 1$,  $l\leq m_1-1$, $ k+j\leq n_1-1$, it
holds that
$$\Big|\partial^k_{h_+}\partial^l_{t_+}\partial^j_\theta R_+\Big|\leq Ch_+^{-\frac{i+1}{2}-k},\ j=0,\ 1,\ \cdots, \ i+1.$$
     \end{lemma}

\begin{remark}  The proof is similar to  the one of Lemma
\ref{P53-1}. We omit it. Without loss of generality, $\alpha_+$ is
still denoted by $\alpha$.

 {It is important to repeat this kind of  transformations till the perturbation is sufficiently small such that it will not affect the final normal form--the weak twist condition (\ref{alphah12}), if the smoothness condition allows us to do so.}
\end{remark}
  With a series of  canonical transformations given in Lemma \ref{L6.2} to eliminate the perturbations by $\nu$ times, the Hamiltonian system with the Hamiltonian $I_3$ can be changed into the  one with the following Hamiltonian
  \begin{equation}\label{I_4}
    \nonumber I_4=\alpha(h_4,t_4)+R_4(h_4, t_4, \theta),
  \end{equation}
  where $\alpha(h_4,t_4)$ satisfies
  \begin{equation}\label{alphah13}
|\partial^k_{h_4}\alpha(h_4,t_4)|\ge   c\cdot
h_4^{\frac{1-d}{2}-k}, \qquad \mbox{for}\ k=0,\ 1,\
  2,
  \end{equation}
   \begin{equation}\label{alphah14}
   \nonumber |\partial^k_{h_4}\partial^l_{t_4}\alpha(h_4,t_4)|\le C\cdot
  h_4^{\frac{1}{2}-k},\  \mbox{for}\  k\le
  \Upsilon_1+1-\nu,\ l\le
  \Upsilon_2-\nu;
  \end{equation}
  and $R_4$ satisfies
  \begin{equation}\label{R_4}
   \nonumber |\partial_{h_4}^k\partial_{t_4}^l\partial_{\theta}^jR_4|\le C\cdot h_4^{-\frac{\nu}{2}-k},\ \mbox{for}\ j\le\nu,\ k+j\le \Upsilon_1-\nu,\ l\le \Upsilon_2-1-\nu.
  \end{equation}

Note that from (\ref{alphah13}),
$|\partial_{h_4}\alpha(h_4,t_4)|\ge c\cdot
h_4^{\frac{-1-d}{2}}>0.$ Thus we can solve the function
$\rho=\alpha(h_4,t_4)$. Denote   $h_4=\mathcal{N}(\rho,t_4)$ be
the inverse function. Exchanging the roles of time and angle
again  { and denoting $(I_4,h_4,t_4)$ } by $(\mathcal{I},h,\tau)$, we
obtain a superlinear Hamiltonian system with the Hamiltonian
\begin{equation}\label{barH_0}
 h(\mathcal{I}, \theta, \tau)=\mathcal{N}(\mathcal{I},
\tau)+\mathcal{P}(\mathcal{I}, \theta, \tau).
\end{equation}
It holds that

 \begin{lemma}\label{L602}
For $\mathcal{\mathcal{I}}$ large enough,  {we have that}
 \begin{eqnarray}\label{N601}
    \nonumber c\mathcal{\mathcal{I}}^2\leq \Big|  \mathcal{N}\Big|\leq
   C\mathcal{\mathcal{I}}^{\frac{2}{1-d}},\
    c\mathcal{I}\leq\Big|\partial_{\mathcal{I}}\mathcal{N}\Big|\leq C\mathcal{I}^{\frac{1+d}{1-d}},\  c\mathcal{I}^{ {-\frac{d}{1-d}}}\leq\Big|\partial^2_{\mathcal{I}}\mathcal{N}\Big|\leq C\mathcal{I}^{\frac{3d}{1-d}};
 \end{eqnarray}
   \begin{eqnarray}\label{N602}
     \nonumber  \Big|\partial^k_\mathcal{I}\partial^l_{\tau }\mathcal{N}\Big|\leq C\mathcal{I}^{\frac{2-k+(2(k+l)-1)d}{1-d}},\quad k\leq\Upsilon_1+1-\nu,\ l\leq\Upsilon_2-\nu.
 \end{eqnarray}
  Moreover, $\mathcal{P}$ satisfies
\begin{equation}\label{mathcalP}
  \nonumber |\partial_{\mathcal{\mathcal{I}}}^k\partial_{\theta}^j\partial_{\tau}^l\mathcal{P}|\le
C\cdot \mathcal{\mathcal{I}}^{\frac{-\nu+1-k+(2(k+l)+1)d}{1-d}}
\end{equation}
\mbox{for}\ $j\le\nu,\ k+j\le \Upsilon_1-\nu,\ l\le
\Upsilon_2-1-\nu.$
     \end{lemma}
    \begin{proof}
 { The proof is similar to the one of Lemma \ref{L14}, we omit it here.}
   \end{proof}

\vskip 0.4cm
{\it Step 5.}\quad \emph{The Poincar\'e map}
\vskip 0.4cm
The Poincar\'e map of the Hamiltonian system with  Hamiltonian
$h$ (\ref{barH_0}) is of the form

\begin{equation}\label{Poin}\left\{ \begin{array}{l}
   \nonumber\theta_{1}=\theta+r(\mathcal{I})+F_1(\mathcal{I},\theta)\\
   \\
    \displaystyle    \mathcal{I}_{1}=\mathcal{I}+F_2(\mathcal{I},\theta),
     \end{array}\right.
      \end{equation}
where $F_1$ and $F_2$ satisfy that for $ j\le\nu-1,\ k+j\le
\Upsilon_1-\nu-1,$
\begin{equation}\label{F1}
   \nonumber|\partial_{
\mathcal{I}}^k\partial_{\theta}^jF_1| \le C\cdot
\mathcal{\mathcal{I}}^{\frac{-\nu -k+(2k+3)d}{1-d}},
\end{equation}
\begin{equation}\label{F2}
  \nonumber   |\partial_{
\mathcal{I}}^k\partial_{\theta}^jF_2| \le  C\cdot
\mathcal{\mathcal{I}}^{\frac{-\nu+1 -k+(2k+1)d}{1-d}}.
\end{equation}
Moreover, the following estimates hold  true for $r(\mathcal{I})$:
\begin{equation}\label{rmathcalI0}
  \nonumber c\mathcal{I}\leq\Big|r(\mathcal{I})\Big|\leq
C\mathcal{I}^{\frac{1+d}{1-d}},
\end{equation}
\begin{equation}\label{rmathcalI1}
\  c\mathcal{I}^{
{-\frac{d}{1-d}}}\leq\Big|r'(\mathcal{I})\Big|\leq
C\mathcal{I}^{\frac{3d}{1-d}};
\end{equation}
and
\begin{equation}\label{rmathcalI}
|r^{(k)}(\mathcal{I})| \leq
C\mathcal{I}^{\frac{1-k+(2k+1)d}{1-d}},\ k\leq\Upsilon_1 -\nu.
\end{equation}
Let $\mathcal{I}(r)$ be the reverse function of $r(\mathcal{I})$. From (\ref{rmathcalI1}) and (\ref{rmathcalI}), we obtain the following estimates on $\mathcal{I}(r)$:
\begin{equation}\label{mathcal{I}(r)1}
  \nonumber c\cdot r^{-\frac{3d}{1+d}} \le \mathcal{I}'(r)\le C\cdot r^{\frac{d}{1-d}}, \end{equation}
and
\begin{equation}\label{mathcal{I}(r)1}
  \nonumber |\mathcal{I}^{(k)}(r)|\le C\cdot
r^{\frac{1-k+(7k-6)d}{1-d}},\quad   k\le \Upsilon_1
-\nu.\end{equation}

With a transformation:$(\theta, \mathcal{I})\rightarrow (\theta, r)$, we obtain the following map:
 \begin{equation}\label{newPoin}\left\{ \begin{array}{l}
 \theta_{1}=\theta+r+\tilde{F}_1(r,\theta)\\
   \\
    \displaystyle    r_{1}=r+\tilde{F}_2(r,\theta),
     \end{array}\right.
      \end{equation}
where
$$
\tilde{F}_1(r,\theta)=F_1(\mathcal{I}(r),\theta),\quad \tilde{F}_2(r,\theta)=\int_0^1r'(\mathcal{I}+\lambda F_2(\mathcal{I},\theta))F_2(\mathcal{I},\theta)d\lambda.
$$

By a direct computation, we have that for $ j\le\nu-1,\ k+j\le
\Upsilon_1-\nu-1,$
\begin{equation}\label{tildeF1}
  \nonumber\begin{array}{ll}|\partial_r^k\partial_{\theta}^j\tilde{F}_1|&\le \sum^k_{i=1}\sum_{k_1+\cdots +k_i=k}|\partial_{\mathcal{I}}^i\partial_{\theta}^jF_1|\cdot |\mathcal{I}^{(k_1)}(r)\cdots
\mathcal{I}^{(k_i)}(r)|\\
\\
&\le  \sum^k_{i=1}C_i\cdot \mathcal{\mathcal{I}}^{\frac{-\nu
-k+(7k-4i+3)d}{1-d}}\le   C \cdot
\mathcal{\mathcal{I}}^{\frac{-\nu -k+(7k-1)d}{1-d}},
\end{array}
\end{equation}
and
\begin{equation}\label{tildeF2}
  \nonumber\begin{array}{ll}
|\partial_r^k\partial_{\theta}^j\tilde{F}_2|\le C\cdot
\mathcal{\mathcal{I}}^{\frac{-\nu+1 -k+7kd}{1-d}}.
\end{array}
\end{equation}

Finally, assume
\begin{equation}\label{para01}
  \nonumber\nu>4,
\end{equation}
\begin{equation}\label{para02}
  \nonumber\Upsilon_1-\nu>4,
\end{equation}
\begin{equation}\label{para02-1}
  \nonumber\Upsilon_2-\nu\ge 1,
\end{equation}
\begin{equation}\label{para03}
  \nonumber-\nu -k+(7k-1)d<0,\ k=0,1,2,3,4,
\end{equation}
and
\begin{equation}\label{para04}
  \nonumber-\nu+1 -k+7kd<0,\ k=0,1,2,3,4.
\end{equation}
Therefore, let {$\nu>\max\{4,7d, 28d-3\}$} and
$\Upsilon_1=5+\nu,\ \Upsilon_2=1+\nu$,   then the map
(\ref{newPoin}) is a standard twist map. Following Moser's
Theorem, we obtain the {bounded} results of Theorem \ref{Th2}.

\subsection{ {Unbounded results for $d>1$}}\label{s6.2}
 In the following, we will prove that   { $\Upsilon_1=\Upsilon_2=4$} are sufficient for the instability results if $d>1$.
\vskip 0.4cm
 {\it Step 1.}\quad\emph{Action-angle variables}
\vskip 0.4cm
 Let $x=\sqrt{\frac{2}{n}}I^{\frac 12}\cos{n\theta}$, $y=\sqrt{\frac{2}{n}}I^{\frac 12}\sin{n\theta}$, then
\begin{equation}\label{}
  \nonumber H_1(I,\theta, t)=I+f_1(I,\theta)+f_2(I,\theta,t).
\end{equation}
  where ${f}_1=\frac 1n
G(\sqrt{\frac{2}{n}}I^{\frac 12}\cos{n\theta})$ and $f_2=-\frac 1n
\sqrt{\frac{2}{n}}I^{\frac 12}\cos{n\theta} p(t).$

  From the condition (\ref{G+-}) on $ g$, we have that $$ G(x)=\int_0^x  g(x)dx :=x\cdot  g(x)+ {f_3(x)}$$ with
  \begin{equation}\label{f_3}
  f_3(x)=O_{ {4}}(|x|^{1-d}),
  \end{equation}
  where $d>1$ and $|x|\gg 1$. Here we say a function $f(x)$ is of $O_{ {m}}(|x|^{c_0})$ for $c_0\in \mathbb{R}$
  if {$|f^{(k)}(x)| \le C |x|^{c_0-k}$ for $x$ satisfying $|x|\gg 1$ and $0\le k\leq m.$} Similarly, for a function $f: \mathbb{R}^+\times \mathbb{S}^2\rightarrow \mathbb{R}$,
  we say $f(I, \theta, t)$ is of $O_{ {m}}(I^{c_0})$ for $c_0\in \mathbb{R}$
  if $ {|\partial^k_I\partial^j_{\theta}\partial^l_t f|}\le C I^{c_0-k}$ for  {$j+k+l\leq m$} and $I\gg 1$.

  Therefore, we get
  $$
  H=I+\widetilde{{f}}_1(I,\theta)+f_2(I,\theta,t)+f_3(\sqrt{\frac{2}{n}}I^{\frac 12}\cos{n\theta}):=I+R(I, \theta, t),
 $$
  where $\widetilde{{f}}_1(I,\theta)=\frac{1}{n}\sqrt{\frac{2}{n}}I^{\frac 12}\cos{n\theta} \cdot g(\sqrt{\frac{2}{n}}I^{\frac 12}\cos{n\theta})$. It holds that $$ {|\partial^k_I\partial^j_{\theta}\partial^l_t \widetilde{f}_1|}\le C I^{\frac{1}{2}-k},\quad {|\partial^k_I\partial^j_{\theta}\partial^l_t f_2|}\le C I^{\frac{1}{2}-k}$$ for $j+k+l\le 4$.

\vskip 0.4cm
 {\it Step 2.}\quad \emph{Some canonical transformations}
\vskip 0.4cm
 Exchanging the roles of angle and time variables, we obtain a new Hamiltonian as follows:

 $$
 \begin{array}{ll}
 {I}&=h-\widetilde{{f}}_1(h-R,\theta)-f_2(h-R, \theta,t)-f_3(\sqrt{\frac{2}{n}}(h-R)^{\frac12}\cos n\theta)\\
 &=h-\widetilde{{f}}_1(h,\theta)-f_2(h,\theta,t)-f_3(\sqrt{\frac{2}{n}}h^{\frac12}\cos n\theta)\\
 &\quad +({\partial_I} \widetilde{{f}}_1(h,\theta)+{\partial_I}f_2(h, \theta,t)+{\partial_I} f_3(\sqrt{\frac{2}{n}}h^{\frac12}\cos n\theta))\cdot R(h, \theta, t)+O_{ {4}}(h^{-\frac 12}).
  \end{array}
$$
 After careful calculations and from the definition of $\widetilde{{f}}_1,\ f_2,\ f_3$, we have that
the term $$({\partial_I} \widetilde{{f}}_1(h,\theta)+{\partial_I}f_2(h, \theta,t)+{\partial_I} f_3(\sqrt{\frac{2}{n}}h^{\frac12}\cos n\theta))\cdot R(h, \theta, t):=R_1(h, \theta, t)+O_{ {4}}(h^{-\frac 12}),$$ { where $R_1(h, \theta, t)$  is of the form $f_4(
g(\sqrt{\frac{2}{n}}h^{\frac 12}\cos n\theta))\cdot f_5(t, \theta)$, with $f_4$
and $f_5$ two smooth functions. In fact,
 $$
 \begin{array}{ll}
 &\Big({\partial_I} \widetilde{{f}}_1(h,\theta)+{\partial_I}f_2(h, \theta,t)+{\partial_I} f_3(\sqrt{\frac{2}{n}}h^{\frac12}\cos n\theta)\Big)\cdot R(h, \theta, t)\\
 =&\Big({\partial_I} \widetilde{{f}}_1(h,\theta)+{\partial_I}f_2(h, \theta,t)+{\partial_I} f_3(\sqrt{\frac{2}{n}}h^{\frac12}\cos n\theta)\Big)\cdot\Big(\widetilde{{f}}_1(h,\theta)+f_2(h,\theta,t)+f_3(\sqrt{\frac{2}{n}}h^{\frac 12}\cos{n\theta})\Big)\\
 =&{\partial_I} \widetilde{{f}}_1\cdot\widetilde{{f}}_1 +{\partial_I}f_2\cdot\widetilde{{f}}_1 +{\partial_I} f_3\cdot\widetilde{{f}}_1
 +{\partial_I} \widetilde{{f}}_1\cdot f_2 +{\partial_I}f_2\cdot f_2 +{\partial_I} f_3\cdot f_2 \\
 &+{\partial_I} \widetilde{{f}}_1\cdot f_3 +{\partial_I}f_2\cdot f_3 +{\partial_I} f_3\cdot f_3, \\
  \end{array}
$$
where
 $$
 \begin{array}{ll}
 &{\partial_I} \widetilde{{f}}_1\cdot\widetilde{{f}}_1= n^{-3}
g^2(\sqrt{\frac{2}{n}}h^{\frac 12}\cos{n\theta})\cdot\cos^2{n\theta}+O_{ {4}}(h^{-\frac 12}),\\
 &{\partial_I}f_2\cdot\widetilde{{f}}_1=n^{-3}
g(\sqrt{\frac{2}{n}}h^{\frac 12}\cos{n\theta})\cdot\cos^2{n\theta}\cdot p(t),\\
 &{\partial_I} \widetilde{{f}}_1\cdot f_2= n^{-3}
g(\sqrt{\frac{2}{n}}h^{\frac 12}\cos{n\theta})\cdot\cos^2{n\theta}\cdot p(t)+O_{ {4}}(h^{-\frac 12}),\\
 &{\partial_I}f_2\cdot f_2=n^{-3}
 \cos^2{n\theta}\cdot p^2(t)+O_{ {4}}(h^{-\frac 12}),\\
 & {\partial_I} f_3\cdot\widetilde{{f}}_1+{\partial_I} f_3\cdot f_2=O_{ {4}}(h^{-1}), \\
 &{\partial_I} \widetilde{{f}}_1\cdot f_3+{\partial_I}f_2\cdot f_3+{\partial_I} f_3\cdot f_3=O_{ {4}}(h^{-\frac 12}),\\
   \end{array}
$$
Denote
\begin{equation}\label{1114}
R_1=n^{-3}\cos^2{n\theta}\Big\{
g^2(\sqrt{\frac{2}{n}}h^{\frac 12}\cos{n\theta})
+2g(\sqrt{\frac{2}{n}}h^{\frac 12}\cos{n\theta}) \cdot p(t)+
    p^2(t)\Big\}.
    \end{equation}
Therefore, we get the conclusion.
Meanwhile, we have }$$ {|\partial^k_h\partial^j_{\theta}\partial^l_t R_1|\le C h^{-k},\quad |\partial^k_h\partial^j_{\theta} f_3|\le C h^{-k}}$$ for $j+k+l\le 4$.
\vskip0.5cm

Next, with the rotation transformation $$
 h=h_1,\qquad
      t=t_1+\theta,
 $$
 the Hamiltonian $I$ above is transformed into
$$
\begin{array}{ll}
 {I_1}&=-\widetilde{{f}}_1(h_1,\theta)-f_2(h_1,\theta,t_1+\theta)-f_3(\sqrt{\frac{2}{n}}h_1^{\frac 12}\cos n\theta))+ R_1(h_1, \theta, t_1+\theta)+O_{ {4}}(h_1^{-\frac 12})\\
 &:=-\widetilde{{f}}_1(h_1,\theta)-f_2(h_1,\theta,t_1+\theta)-f_3(\sqrt{\frac{2}{n}}h_1^{\frac 12}\cos n\theta))+R_2(h_1,\theta, t_1)+O_{ {4}}(h_1^{-\frac 12}),
 \end{array}
$$
where $R_2(h_1,\theta, t_1)=R_1(h_1,\theta, t_1+\theta)$ and   for $ i+j+k\le 4$,
$$ {|\partial^k_{h_1}\partial^j_{\theta}\partial^l_t R_2| \le C h_1^{-k}}.$$
\vskip 0.4cm
{\it Step 3.}\quad \emph{ Normal form}
\vskip 0.4cm
 We first eliminate the terms of $\widetilde{{f}}_1$ and $f_2$ with a generating function $S_1$:
 $$\left\{
 \begin{array}{ll}
 h_1&={ {h_2}}+{\partial_{t_1} S_1},\\
 { {t_2}}&=t_1+{\partial_{ {h_2}}} S_1.\end{array}
 \right.$$
 Let   $S_1 ({ {h_2}}, t_1, \theta)=\int_0^\theta\{
 -\widetilde{{f}}_1({ {h_2}},\theta)-f_2({ {h_2}}, t_1, \theta)+[\widetilde{{f}}_1]({ {h_2}})+[f_2]({ {h_2}}, t_1)\}d\theta$.

 Thus we obtain a Hamiltonian as follows:
 $$
 I_2=-[\widetilde{{f}}_1]({ {h_2}})-[f_2]({ {h_2}}, { {t_2}})- {f_3(\sqrt{\frac{2}{n}}{ {h_2}}^{\frac 12}\cos n\theta)}+R_2({ {h_2}}, { {t_2}}, \theta)+R_3({ {h_2}}, { {t_2}}, \theta)+O_{ {3}}({ {h_2}}^{-\frac 12}),
 $$
where
$$
R_2=-(\frac{\partial \widetilde{{f}}_1}{\partial
{ {h_2}}}+\frac{\partial f_2}{\partial { {h_2}}}+\frac{\partial f_3}{\partial { {h_2}}})\cdot
\frac{\partial S_1}{\partial t_1}-\frac{\partial f_2}{\partial
t_1}\cdot \frac{\partial S_1}{\partial { {h_2}}}+ O_{ {3}}({ {h_2}}^{-\frac 12}).$$

From the definition of $\widetilde{{f}}_1,\ f_2,\ f_3$ and
$S_1$, it follows that $R_2+R_3$ is of the form $f_6(
g(\sqrt{\frac{2}{n}}{ {h_2}}^{\frac 12}\cos n\theta))\cdot f_7({ {t_2}}, \theta)+O_{ {3}}({ {h_2}}^{-\frac 12})$ with $f_6$
and $f_7$ two smooth functions {like $R_1$ (\ref{1114})}. Thus the Hamiltonian can be
rewritten as follows
$$
 I_2=-[\widetilde{{f}}_1]({ {h_2}})-[f_2]({ {h_2}}, { {t_2}})- {f_3(\sqrt{\frac{2}{n}}{ {h_2}}^{\frac 12}\cos n\theta)}+f_6( g(\sqrt{\frac{2}{n}}{ {h_2}}^{\frac 12}\cos n\theta))\cdot f_7({ {t_2}}, \theta)+O_{ {3}}({ {h_2}}^{-\frac 12}).
 $$

Next, to eliminate the term $ {f_6( g(\sqrt{\frac{2}{n}}{ {h_2}}^{\frac 12}\cos n\theta))}\cdot f_7({ {t_2}}, \theta)$, we make the following  canonical transformation:
$$\left\{
 \begin{array}{ll}
 { {h_2}}&={ {\rho}}+\partial_{ {t_2}} S_2,\\
 { {\tau}}&={ {t_2}}+{\partial_{ {\rho}}} S_2.\end{array}
 \right.$$
where the generating function $S_2$ satisfying  $$  S_2 ({ {\rho}}, { {t_2}}, \theta)=\int_0^\theta\{
 -f_6( g(\sqrt{\frac{2}{n}}{ {\rho}}^{\frac 12}\cos n\theta))\cdot f_7({ {t_2}}, \theta)+[f_{67}]({ {\rho}}, { {t_2}})\}d\theta$$ with
 $$[f_{67}]({ {\rho}}, { {t_2}})=\frac{1}{2\pi}\int_0^{2\pi} f_6( g(\sqrt{\frac{2}{n}}{ {\rho}}^{\frac 12}\cos n\theta))\cdot f_7({ {t_2}}, \theta)d\theta.$$

 Hence the obtained Hamiltonian is of the form
 $$
 I_3=-[\widetilde{{f}}_1]({ {\rho}})-[f_2]({ {\rho}}, { {\tau}})-f_3(\sqrt{\frac{2}{n}}{ {\rho}}^{\frac 12}\cos n\theta)+[f_{67}]({ {\rho}}, { {\tau}})+O_{ {2}}({ {\rho}}^{-\frac 12}).
 $$

For the definition of $ g$, we have that

$$
 \begin{array}{ll}
[f_{67}]({ {\rho}}, { {\tau}})&=\frac{1}{2\pi}\left(\int_{T^+} f_6( g(+\infty))\cdot f_7({ {\tau}}, \theta)d\theta+\int_{T^-} f_6( g(-\infty))\cdot f_7({ {\tau}}, \theta)d\theta\right)+O_{ {2}}({ {\rho}}^{-\frac {d}{10}})
\\
  &:=f_8({ {\tau}})+O_{ {2}}({ {\rho}}^{-\frac {d}{10}}).
  \end{array}
$$
where $T^+=\{\theta\in [0, 2\pi]|\cos\theta \ge 0\}$, $T^-=\{\theta\in [0, 2\pi]|\cos\theta < 0\}$ and
 {   $f_8({{\tau}})=O_2(1)$}.
\vskip 0.4cm
Again from the definitions of $ g$, $\widetilde{{f}}_1$ and $f_2$,
we obtain
$$[\widetilde{{f}}_1]({ {\rho}})=\frac{\sqrt{2}}{2\pi}n^{-\frac{3}{2}}A\sqrt{{ {\rho}}
}+O_{ {2}}({ {\rho}}^{-\frac {d}{10}})$$ and
\begin{eqnarray}\label{}
   \nonumber[f_2]({ {\rho}} ,{ {\tau}} )&=&-\frac{\sqrt{2}}{2\pi}n^{-\frac{3}{2}}{ {\rho}} ^{\frac{1}{2}}\{\cos
(n{ {\tau}} )\int_0^{2\pi}p({ {s}})\cos (n{ {s}})d{ {s}}+\sin
(n{ {\tau}} )\int_0^{2\pi}p({ {s}})\sin( n{ {s}})d{ {s}}\}\\
   \nonumber &=&-\frac{\sqrt{2}}{2\pi}n^{-\frac{3}{2}}h ^{\frac{1}{2}}A\cos(n{ {\tau}} +\xi)
 \end{eqnarray}
 with
   $\tan \xi=\frac{\int_0^{2\pi}p(s)\sin( ns)ds}{\int_0^{2\pi}p(s)\cos (ns)ds}.$
\vskip 0.4cm
In conclusion, the new Hamiltonian is of the form
$$
I_3=-\frac{\sqrt{2}}{2\pi}n^{-\frac{3}{2}}A(1-\cos(n{ {\tau}}
+\xi))\sqrt{{ {\rho}} }-f_3(\sqrt{\frac{2}{n}}{ {\rho}}^{\frac 12}\cos n\theta)+
f_8({ {\tau}})+O_{ {2}}({ {\rho}}^{-\frac {d}{10}})+ {O_{ {2}}(\rho^{-\frac 12})}.
$$

Similarly, we can construct a canonical transformation to eliminate the term $f_3(\sqrt{\frac{2}{n}}{ {\rho}}^{\frac 12}\cos n\theta)$ and  obtain the following Hamiltonian
\begin{equation}\label{I10}
  \nonumber I_4=-\frac{\sqrt{2}}{2\pi}n^{-\frac{3}{2}}A(1-\cos(n{ {\tau}}
+\eta))\sqrt{{ {\rho}} }-[f_3]({ {\rho}})+f_8({ {\tau}})+O_{ {1}}({ {\rho}}^{-\frac
{d}{10}})+ {O_{ {1}}(\rho^{-\frac 12})},
 \end{equation}
where $[f_3]({ {\rho}})=\int_0^{2\pi}f_3(\sqrt{\frac{2}{n}}{ {\rho}}^{\frac 12}\cos n\theta)d\theta$. With the help of (\ref{f_3}), we have
\begin{equation}\label{[f_3]}
|[f_3]^{(k)}({ {\rho}})|\le C\cdot { {\rho^{ {-k-\frac{1 }{2}}}}},\quad k=0, 1.
\end{equation}
{ {In fact, for $\rho$ large enough,
\begin{eqnarray*}\label{}
 \displaystyle\Big|[f_3]( {\rho})\Big| &\leq&4\int_0^{\rho^{-\frac{1}{2}}}\Big|f_3(\sqrt{\frac{2}{n}}{{\rho}}^{\frac 12}\sin n\theta)\Big |d\theta +4\int^{\frac{\pi}{2}}_{\rho^{-\frac{1}{2}}}\Big|f_3(\sqrt{\frac{2}{n}}{{\rho}}^{\frac 12}\sin n\theta)\Big|d\theta \\
    \displaystyle &\le& C_1\cdot { {\rho}}^{{-\frac{1}{2}}} +4\int^{\frac{\pi}{2}}_{\rho^{-\frac{1}{2}}} |x |^{d-1}\Big|f_3(x)\Big|\rho^{\frac{1-d}{2}}\sin^{1-d}\theta d\theta  \\
     \displaystyle &\le& C_1\cdot { {\rho}}^{{-\frac{1}{2}}} +C_2 \rho^{\frac{1-d}{2}}\int^{\frac{\pi}{2}}_{\rho^{-\frac{1}{2}}}\sin^{1-d}\theta d\theta \\
      \displaystyle &\le& C_1\cdot { {\rho}}^{{-\frac{1}{2}}} +\frac{2}{\pi}C_2 \rho^{\frac{1-d}{2}}\int^{\frac{\pi}{2}}_{\rho^{-\frac{1}{2}}}\theta^{1-d} d\theta \le  C \cdot { \max\{{\rho^{-\frac{1}{2}},\ {\rho}^{{\frac{1-d}{2}}}}\}},  \\
               \end{eqnarray*}
               where $x=\sqrt{\frac{2}{n}}{{\rho}}^{\frac 12}\sin n\theta$.
               Moreover, the estimate of $[f_3]'(\rho)$ is similar. Therefore, (\ref{[f_3]}) holds. }

\vskip0.5cm

 From (\ref{[f_3]}), we have that
 \begin{equation}\label{I10-1}
h({ {\rho}} ,{ {\tau}}, \theta
)=-\frac{\sqrt{2}}{2\pi}n^{-\frac{3}{2}}A(1-\cos(n{ {\tau}}
+\eta))\sqrt{{ {\rho}} }+f_8({ {\tau}}) + {O_1(\rho^{-\delta})},
 \end{equation}
 { where $-\delta=max\{ {-\frac 12},\ {\frac
{1-d}{2}},\ {\frac
{-d}{10}}\}<0$ for $d>1$}.

 {Finally, we prove that}
  \begin{lemma} The Hamiltonian system with Hamiltonian (\ref{I10-1}) has unbounded solutions.
  \end{lemma}

\begin{proof}
The system with Hamiltonian (\ref{I10-1}) is given by
 \begin{equation}\label{I10-2}\left\{ \begin{array}{l}
   \displaystyle  \frac{d { {\tau}}}{d\theta}=\partial_{ { {\rho}}}h({ {\rho}} ,{ {\tau}}, \theta ),\\
   \\
    \displaystyle  \frac{d { {\rho}}}{d\theta}=-\partial_{{ {\tau}}} h({ {\rho}} ,{ {\tau}}, \theta ).
     \end{array}\right.
      \end{equation}
      And the phase flow is determined by
      \begin{equation}\label{I11}
  \nonumber \frac{d { {\tau}}}{d{ {\rho}}}=\frac{-\frac{\sqrt{2}}{4\pi}n^{-\frac{3}{2}}A(1-\cos(n{ {\tau}} +\eta)){ {\rho}}^{-\frac{1}{2}}+O_{ {0}}({ {\rho}}^{-\delta-1})}{\frac{\sqrt{2}}{2\pi}n^{-\frac{1}{2}}A\sin(n{ {\tau}} +\eta){ {\rho}}^{\frac{1}{2}}-f_8'({ {\tau}})+O_{ {0}}({ {\rho}}^{-\delta})}.
 \end{equation}

 Assume ${ {\tau}}^*$ satisfying
$1-\cos(n{ {\tau}}^* +\xi)=0,$
then, as ${ {\tau}}\rightarrow { {\tau}}^*,$
$$1-\cos(n{ {\tau}} +\xi)=\frac{1}{2}n^2({ {\tau}}-{ {\tau}}^*)^2+O(|{ {\tau}}-{ {\tau}}^*|^4),$$
 $$\sin(n{ {\tau}} +\xi)=n({ {\tau}}-{ {\tau}}^*)+O(|{ {\tau}}-{ {\tau}}^*|^3).$$
Denote $\zeta={ {\tau}}-{ {\tau}}^*$,  note that $-\delta<0$, we have

  \begin{equation}\label{I13}
 (-\frac{1}{4}-\frac {\delta}{4})\zeta { {\rho}}^{-1}- C\zeta^{-1}{ {\rho}}^{-\frac 32-\delta}\leq\frac{d \zeta}{d{ {\rho}}}\leq-{\frac{1}{5}}\zeta { {\rho}}^{-1}+ C \zeta^{-1}{ {\rho}}^{-\frac 32-\delta}
 \end{equation}
 in the domain $$D=\big\{(\zeta,{ {\rho}})\big|{ {\rho}}^{-\frac{5}{12}}\leq\zeta\leq { {\rho}}^{-\frac{1}{20}}\big\}.$$
 Solving the Bernoulli equation (\ref{I13}), we have
  \begin{equation}\label{I14}
  \nonumber c_1{ {\rho}}^{-\frac{1}{2}-\frac {\delta}{2}}+c_2{ {\rho}}^{-\frac 12-\delta}\leq\zeta^2\leq C_1{ {\rho}}^{-\frac{2}{5}}+C_2{ {\rho}}^{-\frac 12-\delta}. \end{equation}
It implies that the phase curve starting from the initial point $(\zeta(0),{ {\rho}}(0))=({ {\rho}}^{-\frac{1}{4}}(0),{ {\rho}}(0))$ with ${ {\rho}}(0)$ large enough stays in the domain $D$.

 Finally, from (\ref{I10-1}) and (\ref{I10-2}), the derivative
$$ \frac{d { {\rho}}}{d\theta}\geq c { {\rho}}^{\frac{1}{12}}\ \  \mbox{in\ domain}\ D,$$
which yields that the curve we obtained above is unbounded, i.e., ${ {\rho}}(\theta)$ goes to infinity as $\theta\rightarrow+\infty.$
\end{proof}

 {Go back to the original equation (\ref{semi}), we have obtained the unbounded solutions of equation (\ref{semi}) for $d>1$.}

\vskip 2cm

\section{Appendix}
\subsection{Pan and Yu's method}\label{s7.1}




\begin{lemma}\label{A03}
{\rm{(\cite{Pan} pp. 226-230, Theorem 10)}} Suppose $0<\lambda<1,\ 0<\mu<1$,
$\varphi(x)\in C^{\nu}[a,b]$, $f(x)\in C^{\nu}[a,b]$ satisfying
\begin{eqnarray}\label{}
  \nonumber f'(x)=(x-a)^{\rho-1}(b-x)^{\sigma-1}f_1(x)
 \end{eqnarray}
where $\rho\geq1,\ \sigma\geq1$, $f_1(x)>0,\ x\in[a,b]$, then
\begin{eqnarray}\label{}
\nonumber  I&=&\int_a^b\varphi(x)e^{inf(x)}(x-a)^{\lambda-1}(b-x)^{\mu-1}dx\\
   \nonumber &=&B(n)-A(n),
       \end{eqnarray}
       where
      $$ A(n)= A_\nu(n)+R_\nu(n),\ B(n)= B_\nu(n)+Q_\nu(n),$$
      and
      $$A_\nu(n)=-\sum_{k=0}^{\nu-1}\frac{h^{(k)}(0)}{k!\rho}\Gamma(\frac{k+\lambda}{\rho})e^{\frac{i(k+\lambda)\pi}{2\rho}}n^{-\frac{k+\lambda}{\rho}}e^{inf(a)},$$
  $$\Big|R_\nu(n)\Big|\leq Cn^{-\frac{\nu}{\rho}},$$
    $$B_\nu(n)=-\sum_{k=0}^{\nu-1}\frac{l^{(k)}(0)}{k!\sigma}\Gamma(\frac{k+\mu}{\sigma})e^{-\frac{i(k+\mu)\pi}{2\sigma}}n^{-\frac{k+\mu}{\sigma}}e^{inf(b)},$$
    $$\Big|Q_\nu(n)\Big|\leq Cn^{-\frac{\nu}{\sigma}}.$$
   \end{lemma}

\begin{remark}\label{A03-1}
When $\lambda=\mu=1,$ then
\begin{eqnarray}\label{}
\nonumber  I&=&\int_a^b\varphi(x)e^{inf(x)}dx\sim B_{\nu-1}(n)-A_{\nu-1}(n).
        \end{eqnarray}
   \end{remark}

\subsection{Proof of Lemma \ref{L05}}\label{s7.2}
\begin{proof} Suppose $ k+j\leq\Upsilon_1+1$ and $ l\leq {\Upsilon_2}$.
\vskip0.3cm
   i)When $k+j+l=0$, the conclusion follows from Lemmas \ref{L01}, \ref{L03} and \ref{L04}.
\vskip0.3cm
   ii)When $k+j+l=1$, define
   \begin{eqnarray}\label{}
   \nonumber g_1(h,t,\theta)&=&\partial_If_1(h-R,\theta)+\partial_If_2(h-R,t,\theta)+\partial_If_3(h-R,\theta);\\
    \nonumber g_2(h,t,\theta)&=&\partial_tf_2(h-R,t,\theta);\\
    \nonumber g_3(h,t,\theta)&=&\partial_\theta f_1(h-R,\theta)+\partial_\theta f_2(h-R,t,\theta)+\partial_\theta f_3(h-R,\theta);\\
      \nonumber \Delta (h,t,\theta)&=&1+\partial_If_1(h-R,\theta)+\partial_If_2(h-R,t,\theta)+\partial_If_3(h-R,\theta). \end{eqnarray}

      Obviously, $\Delta (h,t,\theta)\geq1/2$ for $h\gg1$ and
       \begin{eqnarray}\label{delta}
  \nonumber   \Delta\cdot \partial_hR(h,t,\theta)=  g_1(h,t,\theta),\ \ \ \Delta\cdot \partial_tR(h,t,\theta)=  g_2(h,t,\theta),\ \ \ \Delta\cdot \partial_\theta R(h,t,\theta)=  g_3(h,t,\theta). \\
   \end{eqnarray}
  From  Lemmas \ref{L01}-\ref{L04}, we obtain
    \begin{eqnarray}\label{}
   \nonumber \frac{1}{2}\Big|\partial_hR(h,t,\theta)\Big|&\leq&\Big|\Delta\cdot \partial_hR(h,t,\theta)\Big|\\
   \nonumber &=&\Big| \partial_If_1(h-R,\theta)+\partial_If_2(h-R,t,\theta)+\partial_If_3(h-R,\theta)\Big|\\
   \nonumber  &\leq&C(h-R)^{-\frac{1}{2}}\leq Ch^{-\frac{1}{2}},\\
  \nonumber \\
   \nonumber \frac{1}{2}\Big|\partial_tR(h,t,\theta)\Big|&\leq&\Big|\Delta\cdot \partial_tR(h,t,\theta)\Big|\\
   \nonumber &=&\Big| \partial_tf_2(h-R,t,\theta)\Big|\\
   \nonumber  &\leq&C(h-R)^{\frac{1}{2}}\leq Ch^{\frac{1}{2}},
    \end{eqnarray}
 and
     \begin{eqnarray}\label{}
   \nonumber \frac{1}{2}\Big|\partial_\theta R(h,t,\theta)\Big|&\leq&\Big|\Delta\cdot \partial_\theta R(h,t,\theta)\Big|\\
   \nonumber &=&\Big|\partial_\theta f_1(h-R,\theta)+\partial_\theta f_2(h-R,t,\theta)+\partial_\theta f_3(h-R,\theta)\Big|\\
   \nonumber  &\leq&C(h-R)^{\frac{1}{2}}\leq Ch^{\frac{1}{2}}.
    \end{eqnarray}
\vskip0.3cm
     iii)When $k+j+l=2$, From i) and ii), we have
       \begin{eqnarray}\label{}
   \nonumber \Big|\partial_hg_1(h,t,\theta)\Big|&\leq&Ch^{-1},\ \ \Big|\partial_tg_1(h,t,\theta)\Big|\leq Ch^{-\frac{1}{2}},\ \ \Big|\partial_\theta g_1(h,t,\theta)\Big| \leq Ch^{0};\\
    \nonumber \Big|\partial_hg_2(h,t,\theta)\Big|&\leq&Ch^{-\frac{1}{2}},\ \ \Big|\partial_tg_2(h,t,\theta)\Big|\leq Ch^{\frac{1}{2}},\ \ \ \Big|\partial_\theta g_2(h,t,\theta)\Big| \leq Ch^{\frac{1}{2}};\\
     \nonumber \Big|\partial_hg_3(h,t,\theta)\Big|&\leq&Ch^{0},\ \ \ \ \Big|\partial_tg_3(h,t,\theta)\Big|\leq Ch^{\frac{1}{2}},\ \ \ \Big|\partial_\theta g_3(h,t,\theta)\Big| \leq Ch^{1};
       \end{eqnarray}
and
  \begin{eqnarray}\label{}
       \nonumber \Big|\partial_h\Delta(h,t,\theta)\Big|&\leq&Ch^{-1},\ \ \ \ \Big|\partial_t\Delta(h,t,\theta)\Big|\leq Ch^{-\frac{1}{2}},\ \ \Big|\partial_\theta \Delta(h,t,\theta)\Big| \leq
       Ch^{0}.
         \end{eqnarray}
         From (\ref{delta}), differentiating on both sides of the
         equations, we obtain:
           \begin{eqnarray}\label{}
   \nonumber  \Delta\cdot \partial^2_hR(h,t,\theta)&=& \partial_h g_1(h,t,\theta)-\partial_h\Delta\cdot \partial_hR(h,t,\theta),\\
   \nonumber  \Delta\cdot \partial^2_tR(h,t,\theta)&=& \partial_t g_2(h,t,\theta)-\partial_t\Delta\cdot \partial_tR(h,t,\theta),\\
  \nonumber  \Delta\cdot \partial^2_\theta R(h,t,\theta)&=& \partial_\theta g_3(h,t,\theta)-\partial_\theta\Delta\cdot \partial_\theta R(h,t,\theta),\\
   \nonumber  \Delta\cdot \partial_{h}\partial_{t}R(h,t,\theta)&=& \partial_t g_1(h,t,\theta)-\partial_t\Delta\cdot \partial_hR(h,t,\theta),\\
   \nonumber  \Delta\cdot \partial_{h}\partial_{\theta}R(h,t,\theta)&=& \partial_\theta g_1(h,t,\theta)-\partial_\theta\Delta\cdot \partial_hR(h,t,\theta),\\
   \nonumber  \Delta\cdot \partial_{t}\partial_{\theta}R(h,t,\theta)&=& \partial_\theta g_2(h,t,\theta)-\partial_\theta\Delta\cdot \partial_tR(h,t,\theta).
    \end{eqnarray}
It follows that
 \begin{eqnarray}\label{}
    \nonumber \frac{1}{2}\Big|\partial^2_hR(h,t,\theta)\Big|&\leq&\Big|\partial_h g_1(h,t,\theta)\Big|+\Big|\partial_h\Delta\cdot \partial_hR(h,t,\theta)  \Big|\\
     \nonumber &\leq&Ch^{-1},
       \end{eqnarray}
        \begin{eqnarray}\label{}
    \nonumber \frac{1}{2}\Big|\partial^2_tR(h,t,\theta)\Big|&\leq&\Big|\partial_t g_2(h,t,\theta)\Big|+\Big|\partial_t\Delta\cdot \partial_tR(h,t,\theta)  \Big|\\
     \nonumber &\leq&Ch^{\frac{1}{2}},
       \end{eqnarray}
        \begin{eqnarray}\label{}
    \nonumber \frac{1}{2}\Big|\partial^2_\theta R(h,t,\theta)\Big|&\leq&\Big|\partial_\theta  g_3(h,t,\theta)\Big|+\Big|\partial_\theta \Delta\cdot \partial_\theta R(h,t,\theta)  \Big|\\
     \nonumber &\leq&Ch^{1},
       \end{eqnarray}
         \begin{eqnarray}\label{}
    \nonumber \frac{1}{2}\Big|\partial_{h}\partial_{t}R(h,t,\theta)\Big|&\leq&\Big|\partial_t  g_1(h,t,\theta)\Big|+\Big|\partial_t \Delta\cdot \partial_h R(h,t,\theta)  \Big|\\
     \nonumber &\leq&Ch^{-\frac{1}{2}},
       \end{eqnarray}
        \begin{eqnarray}\label{}
    \nonumber \frac{1}{2}\Big|\partial_{h}\partial_{\theta}R(h,t,\theta)\Big|&\leq&\Big|\partial_\theta  g_1(h,t,\theta)\Big|+\Big|\partial_\theta \Delta\cdot \partial_h R(h,t,\theta)  \Big|\\
     \nonumber &\leq&Ch^{0},
       \end{eqnarray}
        \begin{eqnarray}\label{}
    \nonumber \frac{1}{2}\Big|\partial_{t}\partial_{\theta}R(h,t,\theta)\Big|&\leq&\Big|\partial_\theta  g_2(h,t,\theta)\Big|+\Big|\partial_\theta \Delta\cdot \partial_t R(h,t,\theta)  \Big|\\
     \nonumber &\leq&Ch^{\frac{1}{2}}.
       \end{eqnarray}

       Generally, if
      $$ \Big|\partial^k_h\partial^l_t\partial^j_\theta  R(h,t,\theta)\Big|\leq Ch^{\frac{1}{2}-\frac{k}{2}+\frac{1}{2}(max\{1,j\}-1)},\ \mbox{for}\ 1\leq j+k+l\leq m,$$
      then
       $$\Big |\partial^k_h\partial^l_t\partial^j_\theta  g_1(h,t,\theta)\Big|\leq Ch^{-\frac{1}{2}-\frac{k}{2}+\frac{j}{2}},$$
 $$\Big |\partial^k_h\partial^l_t\partial^j_\theta  g_2(h,t,\theta)\Big|\leq Ch^{\frac{1}{2}-\frac{k}{2}},$$
  $$\Big |\partial^k_h\partial^l_t\partial^j_\theta  g_3(h,t,\theta)\Big|\leq Ch^{\frac{1}{2}-\frac{k}{2}+\frac{j}{2}},$$
   $$\Big |\partial^k_h\partial^l_t\partial^j_\theta  \Delta(h,t,\theta)\Big|\leq Ch^{-\frac{1}{2}-\frac{k}{2}+\frac{j}{2}},$$
The proof of these estimates is based on Leibniz's rule and a direct computation. Consequently,
by induction and Leibniz's rule again to (\ref{delta}), we obtain
 $$ \Big|\partial^k_h\partial^l_t\partial^j_\theta  R(h,t,\theta)\Big|\leq Ch^{\frac{1}{2}-\frac{k}{2}+\frac{1}{2}(max\{1,j\}-1)},\ \mbox{for}\ 1\leq j+k+l\leq m+1.$$
 \end{proof}

\subsection{Proof of Lemma \ref{L06}}

  \begin{proof} The estimates on $R_{01}$ are based on a direct
    computation and Lemmas \ref{L01} and \ref{L03}.

    The estimates on $R_{02}$ are based on  Leibniz's rule, Lemmas
    \ref{L01}--\ref{L05}, and the following claim. Readers can also refer to \cite{Liu01} (Lemma 3.4).
\vskip0.3cm
    \noindent Claim. For $h$ large enough, $\theta,\ t\in \mathbb{S}^1$,  $ k+j\leq\Upsilon_1-1$ and $ l\leq {\Upsilon_2}$, it holds that:
\begin{equation}\label{R02-1}
  \Big|\partial^k_h\partial^l_t\partial^j_\theta \partial_If_1(h-s R,\theta)\Big|\leq Ch^{-\frac{1}{2}-\frac{k}{2}+\frac{1}{2}(max\{1,j\}-1)};
 \end{equation}
 \begin{equation}\label{R02-2}
  \Big|\partial^k_h\partial^l_t\partial^j_\theta \partial_I^2f_1(h-s R,\theta)\Big|\leq Ch^{-\frac{3}{2}-\frac{k}{2}+\frac{1}{2}(max\{1,j\}-1)};
 \end{equation}
 \begin{equation}\label{R02-3}
 \Big|\partial^k_h\partial^l_t\partial^j_\theta \partial_If_2(h-s R,\theta)\Big|\leq
 Ch^{-\frac{1}{2}-\frac{k}{2}+\frac{1}{2}(max\{1,j\}-1)};
 \end{equation}
 \begin{equation}\label{R02-4}
 \Big|\partial^k_h\partial^l_t\partial^j_\theta \partial_I^2f_2(h-s R,\theta)\Big|\leq
 Ch^{-\frac{3}{2}-\frac{k}{2}+\frac{1}{2}(max\{1,j\}-1)};
 \end{equation}and
 \begin{equation}\label{R02-5}
 \Big|\partial^k_h\partial^l_t\partial^j_\theta f_3(h- R,\theta)\Big|\leq Ch^{-\frac{k}{2}+\frac{1}{2}(max\{1,j\}-1)}.
 \end{equation}
 \emph{Proof of (\ref{R02-1})}.

  When k+l+j=0, then
 \begin{equation}\label{}
 \nonumber \Big|\partial_If_1(h-s R,\theta)\Big|\leq Ch^{-\frac{1}{2}}.
 \end{equation}

 For $k+l+j>0$, using Leibniz's rule, $\partial^k_h\partial^l_t\partial^j_\theta f_1(h-s R,\theta)$ is the sum of terms
 $$
(\partial_{I}^{u}\partial_{\theta}^{v} \partial_If_1)\cdot
\Pi_{i=1}^{
u}\partial_{h_3}^{k_i}\partial_{t_3}^{l_i}\partial_{\theta}^{j_i}
(h-s R),$$ with $1\leq u+v\leq j+k+l,\ \sum_{i=1}^{u}k_i =k,\
\sum_{i=1}^{u}l_i =l,\ v+\sum_{i=1}^{u}j_i =j,\ and \
k_i+{j_i}+{l_i}\geq1,\ i=1,\ldots,u.$ Thus from Lemmas
\ref{L01} and \ref{L05}, it holds that
$$\Big|\partial^k_h\partial^l_t\partial^j_\theta \partial_If_1(h-s
R,\theta)\Big|\leq
Ch^{-\frac{1}{2}-\frac{k}{2}+\frac{1}{2}(max\{1,j\}-1)}.$$

The proofs of (\ref{R02-2}--\ref{R02-5}) are similar to the one of
(\ref{R02-1}). We omit it here.  \end{proof}

\subsection{Proof of Lemma \ref{L08}}
\begin{proof}
(\ref{S02-1}) follows from (\ref{S02}) and Lemma \ref{L01}.

From (\ref{T02}), it is easy to see
$$|\partial_{{h_2}} {t_1}|\leq C{h_2}^{-\frac{3}{2}},\ \
\partial_{{t_2}} {t_1}=1,\ \ |\partial_{\theta}
{t_1}|\leq C{h_2}^{-\frac{1}{2}}.$$

By a direct computation, for $k+l+j\geq2$ and $k+j\leq\Upsilon_1$,
\begin{eqnarray}\label{phi2-1}
  \nonumber |\partial^k_{{h_2}}\partial^l_{{t_2}}\partial^j_\theta
{t_1}|&=&|\partial^k_{{h_2}}\partial^l_{{t_2}}\partial^j_\theta(t_2-\partial_{
h_2}S_2(h_2,\theta))|\\
  \nonumber &=&|\partial^{k+1}_{{h_2}}\partial^l_{{t_2}}\partial^j_\theta
  S_2(h_2,\theta)|\\
  \nonumber   &\leq& C{h_2}^{-\frac{1}{2}-k+\frac{1}{2}(max\{2,j\}-2)}.
 \end{eqnarray}

Next, we consider the estimates on $R_{21}$. First, it holds that
$|R_{21}|\leq C$.

Suppose $k+j\leq\Upsilon_1-1.$
\vskip0.3cm
i) Consider
$\partial^k_{h_2}\partial^l_{t_2}\partial^j_\theta
R_{11}(h_2,t_2-\partial_{h_2}S_2(h_2,\theta),\theta)$. From (\ref{R021}) and By Leibniz's
rule, it is the  summation of terms $$
(\partial_{h_1}^{p}\partial_{t_1}^{q}\partial_{\theta}^{r} R_{11})
(\Pi_{i=1}^{
q}\partial_{h_2}^{k_i}\partial_{t_2}^{l_i}\partial_{\theta}^{j_i}
t_1) $$ with $1\leq p+q+r\leq j+k+l,\ p+\sum_{i=1}^{q}k_i =k,\
\sum_{i=1}^{q}l_i =l,\ r+\sum_{i=1}^{q}j_i =j,\ \mbox{and} \
k_i+{j_i}+{l_i}\geq1,\ i=1,\ldots,q,$ which implies that
$$|\partial^k_{h_2}\partial^l_{t_2}\partial^j_\theta R_{11}(h_2,t_2-\partial_{h_2}S_2(h_2,\theta),\theta)|\leq
C{h_2}^{-k+\frac{1}{2}(max\{1,j\}-1)},\ \mbox{for}\
l\leq{\Upsilon_2}.$$
\vskip0.3cm
ii) Similar to the part i),
 with Lemma \ref{L03}  we have
$$|\partial^k_{h_2}\partial^l_{t_2}\partial^j_\theta
(\partial_{t_1}f_2(h_2,\theta,t_2+\theta-\mu\partial_{h_2}S_2))|\leq
C{h_2}^{-k+\frac{1}{2}(max\{2,j\}-2)},\ \mbox{for}\
l\leq{\Upsilon_2}-1.$$ By Leibniz's rule and the estimates on
$\partial^k_{h_2}\partial^l_{t_2}\partial^j_\theta
S_2(h_2,\theta)$, it holds that
$$|\partial^k_{h_2}\partial^l_{t_2}\partial^j_\theta
(\partial_{t_1}f_2(h_2,\theta,t_2-\mu\partial_{h_2}S_2)\partial_{h_2}S_2)|\leq
C{h_2}^{-\frac{1}{2}-k+\frac{1}{2}(max\{2,j\}-2)},\ \mbox{for}\
l\leq{\Upsilon_2}-1,$$
which, together with parts i) and ii), implies that
$$|\partial^k_{h_2}\partial^l_{t_2}\partial^j_\theta R_{21}|\leq
C{h_2}^{-k+\frac{1}{2}(max\{1,j\}-1)},\ \mbox{for}\
l\leq{\Upsilon_2}-1.$$

The proof of $R_{22}$ is similar to the one of $R_{21}$, we omit it. \end{proof}

\subsection{Proof of Lemma \ref{LQ1}}
\begin{proof}
$Q=h_3^{-1}(h_2-h_3)-h_3^{-1}R$ implies that
     \begin{eqnarray}\label{}
\nonumber\Big|Q \Big|\leq C{h_3}^{-\frac{1}{2}}.
 \end{eqnarray}

Now we consider the estimates on derivatives.

 Suppose that
$l\leq{\Upsilon_2}-1,\  k+j\leq\Upsilon_1.$ Using Leibniz's rules, $
   \partial^k_{h_3}\partial^l_{t_3}\partial^j_{\theta}(h_3^{-1}h_2)$  is the summation of terms
     $\partial^{k_1}_{h_3}h_3^{-1}\cdot\partial^{k_2}_{h_3}\partial^l_{t_3}\partial^j_{\theta}h_2$, where
     $h_2=h_2(h_3,t_3,\theta).
  $
     Following Lemma \ref{L09}, we have
     \begin{eqnarray}\label{L091}
 \Big|\partial^k_{{h_3}}\partial^l_{{t_3}}\partial^j_\theta
{(h_3^{-1}h_2)}\Big|\leq C{h_3}^{-\frac{1}{2}-k},\ k+l+j\geq1.
 \end{eqnarray}

Next, consider
$\partial^k_{{h_3}}\partial^l_{{t_3}}\partial^j_\theta
(h_3^{-1}R(h_3,t_3,\theta)).$ Note that
$h=h_1=h_2=h_2(h_3,t_3,\theta)$ (see Lemma \ref{L09} for the
estimates), and $t=t(h_3,t_3,\theta)$. We first estimate $t=t(h_3,t_3,\theta)$, which can be regarded as the composition of $t=t(h_2,t_2,\theta)$ and $h_2=h_2(h_3,t_3,\theta),\ t_2=t_2(h_3,t_3,\theta).$

\vskip0.3cm
\emph{Step 1.}\  Consider $t=t(h_2,t_2,\theta)$ be the composition of
$t=t_1+\theta$, $t_1=t_1(h_2,t_2,\theta).$

 Following Leibniz's rule, $\partial^k_{{h_2}}\partial^l_{{t_2}}\partial^j_\theta
{t}(h_2,t_2,\theta)$ is the summation of terms
$$
(\partial_{t_1}^{q}\partial_{\theta}^{r} t) (\Pi_{i=1}^{
q}\partial_{h_2}^{k_i}\partial_{t_2}^{l_i}\partial_{\theta}^{j_i}
t_1),$$ with $1\leq q+r\leq j+k+l,\ \sum_{i=1}^{q}k_i =k,\
\sum_{i=1}^{q}l_i =l,\ r+\sum_{i=1}^{q}j_i =j,\ and \
k_i+{j_i}+{l_i}\geq1,\ i=1,\ldots,q.$

 Following  Lemma \ref{L08}, we have
 \begin{eqnarray}\label{t2}
 \nonumber  \Big|\partial_{{h_2}} {t}\Big|\leq C{h_2}^{-\frac{3}{2}},\ \
\frac{1}{2} \leq\Big|\partial_{{t_2}} {t}\Big|\leq 2,\ \
\frac{1}{2} \leq\Big|\partial_{\theta} {t}\Big|\leq  2,\\
 \Big|\partial^k_{{h_2}}\partial^l_{{t_2}}\partial^j_\theta
{t}\Big|\leq C{h_2}^{-\frac{1}{2}-k+\frac{1}{2}(max\{2,j\}-2)},\
k+l+j\geq2.
 \end{eqnarray}
\vskip0.3cm
 \emph{Step 2.}\  Consider $t=t(h_3,t_3,\theta)$ be the composition of $t=t(h_2,t_2,\theta)$, $h_2=h_2(h_3,t_3,\theta),\ t_2=t_2(h_3,t_3,\theta).$
 Following Leibniz's rule, $\partial^k_{{h_3}}\partial^l_{{t_3}}\partial^j_\theta
{t}$ is the summation of terms
$$
(\partial_{h_2}^{p}\partial_{t_2}^{q}\partial_{\theta}^{r} t)
(\Pi_{i=1}^{ p}\partial_{h_3}^{k_i}\partial_{t_3}^{l_i}\partial_{\theta}^{j_i} h_2)(\Pi_{i=p+1}^{ p+q}\partial_{h_3}^{k_i}\partial_{t_3}^{l_i}\partial_{\theta}^{j_i} t_2),$$
with $1\leq p+q+r\leq j+k+l,\ \sum_{i=1}^{p+q}k_i =k,\ \sum_{i=1}^{p+q}l_i =l,\
r+\sum_{i=1}^{p+q}j_i =j,\ and \ k_i+{j_i}+{l_i}\geq1,\ i=1,\ldots,p+q.$

 With Lemma \ref{L09} and the estimates (\ref{t2}),  we obtain directly the estimates on the function $ t(h_3,t_3,\theta)$ as following:
 \begin{eqnarray}\label{t3}
 \nonumber  \Big|\partial_{{h_3}} {t}\Big|\leq C{h_3}^{-\frac{3}{2}},\ \
c \leq\Big|\partial_{{t_3}} {t}\Big|\leq C,\ \
c \leq\Big|\partial_{\theta} {t}\Big|\leq  C,\\
 \Big|\partial^k_{{h_3}}\partial^l_{{t_3}}\partial^j_\theta
{t}\Big|\leq C{h_3}^{-\frac{1}{2}-k+\frac{1}{2}(max\{2,j\}-2)},\
k+l+j\geq2.
 \end{eqnarray}
\vskip0.4cm
 Then, let us consider $\partial^k_{{h_3}}\partial^l_{{t_3}}\partial^j_\theta R(h_3,t_3,\theta).$ By Leibniz's rule, it is the  summation of terms
$$
(\partial_{h}^{p}\partial_{t}^{q}\partial_{\theta}^{r} R)
(\Pi_{i=1}^{ p}\partial_{h_3}^{k_i}\partial_{t_3}^{l_i}\partial_{\theta}^{j_i} h_2)(\Pi_{i=p+1}^{ p+q}\partial_{h_3}^{k_i}\partial_{t_3}^{l_i}\partial_{\theta}^{j_i} t),$$
with $1\leq p+q+r\leq j+k+l,\ \sum_{i=1}^{p+q}k_i =k,\ \sum_{i=1}^{p+q}l_i =l,\
r+\sum_{i=1}^{p+q}j_i =j,\ and \ k_i+{j_i}+{l_i}\geq1,\ i=1,\ldots,p+q.$

 With Lemmas \ref{L05}, \ref{L09} and the estimates (\ref{t3}),  we obtain directly the estimates as following:
 $$\Big|\partial^k_{{h_3}}\partial^l_{{t_3}}\partial^j_\theta R(h_3,t_3,\theta)\Big|\leq Ch_3^{\frac{1}{2}-\frac{k}{2}+\frac{1}{2}(max\{1,j\}-1)},\ l\leq{\Upsilon_2}-1,\  k+j\leq\Upsilon_1.$$

Therefore,
 \begin{eqnarray}\label{L092}
   \nonumber \Big|\partial^k_{{h_3}}\partial^l_{{t_3}}\partial^j_\theta (h_3^{-1}R)\Big|\leq Ch_3^{-\frac{1}{2}-\frac{k}{2}+\frac{1}{2}(max\{1,j\}-1)},\ l\leq{\Upsilon_2}-1,\  k+j\leq\Upsilon_1,
 \end{eqnarray}
which, together with (\ref{L091}), implies that
\begin{eqnarray}\label{}
  \nonumber \Big|\partial^k_{{h_3}}\partial^l_{{t_3}}\partial^j_\theta Q\Big|\leq Ch_3^{-\frac{1}{2}-\frac{k}{2}+\frac{1}{2}(max\{1,j\}-1)},\ l\leq{\Upsilon_2}-1,\  k+j\leq\Upsilon_1.
 \end{eqnarray}

\vskip0.4cm
 Finally, we consider the expression of $\partial^2_\theta Q$. Note that
 $Q=h_3^{-1}(h_2-h_3)-h_3^{-1}R$.
We have
\begin{eqnarray}\label{}
\nonumber \partial_\theta(h_2-h_3)&=&\partial_\theta(\partial_{ t_2}S_3(h_3,t_2,\theta))\\
\nonumber&=& \partial^2_{ t_2}S_3(h_3,t_2,\theta)t_{2,\theta}+\partial_{ \theta}\partial_{ t_2}S_3(h_3,t_2,\theta),
 \end{eqnarray}
 with $t_{2,\theta}=\partial_\theta t_2$, and thus,
\begin{eqnarray}\label{}
\nonumber \partial^2_\theta(h_2-h_3)&=&\partial_\theta(\partial^2_{ t_2}S_3(h_3,t_2,\theta)t_{2,\theta}+\partial_{ \theta}\partial_{ t_2}S_3(h_3,t_2,\theta))\\
\nonumber&=&\partial^3_{ t_2}S_3(h_3,t_2,\theta)(t_{2,\theta})^2+2\partial_\theta\partial^2_{ t_2}S_3(h_3,t_2,\theta)t_{2,\theta}+ \partial^2_{ t_2}S_3(h_3,t_2,\theta)t_{2,\theta\theta}\\
\nonumber&&+\partial^2_\theta\partial_{ t_2}S_3(h_3,t_2,\theta).
 \end{eqnarray}
Then, from Lemma \ref{L09},  it holds that
\begin{eqnarray}\label{h23}
 \Big|\partial^2_\theta(h_2-h_3 )\Big|\leq C
h_3^{\frac{1}{2}}.
 \end{eqnarray}

Recall
\begin{equation*}\label{}
    R(h_3,t_3,\theta)=f_1(h,\theta)+f_2(h,t,\theta)+\frac{1}{n}\psi(x)-R_{01}(h,t,\theta)-R_{02}(h,t,\theta),
 \end{equation*}
where
$f_1(h,\theta)=\frac{1}{n}G(\sqrt{\frac{2}{n}}h^{\frac{1}{2}}\cos{n\theta}),\
f_2(h,\theta,t)=-\frac{1}{n}\sqrt{\frac{2}{n}}h^{\frac{1}{2}}\cos{n\theta}
p(t),\ h=h_2=h_2(h_3,t_3,\theta),$ and $t=t(h_3,t_3,\theta)$. The
estimates of $R$ are divided into the following five parts:
\vskip0.3cm
(i) Consider $f_1(h,\theta)=\frac{1}{n}G(x)$ with
$x=\sqrt{\frac{2}{n}}h^{\frac{1}{2}}\cos{n\theta}$.
 We have $$n\partial_\theta\Big(
 f_1\Big|_{(h_3,t_3,\theta)}\Big)=G'(x)(x_hh_{2,\theta}+x_\theta),$$
and
\begin{eqnarray}\label{rr01}
\nonumber n\partial^2_\theta\Big(
 f_1\Big|_{(h_3,t_3,\theta)}\Big)&=&G''(x)(x_hh_{2,\theta}+x_\theta)^2\\
\nonumber
&&+G'(x)(x_{hh}(h_{2,\theta})^2+2x_{h\theta}h_{2,\theta}+x_{h}h_{2,\theta\theta}+x_{\theta\theta})\\
&=&g_1(h_3,t_3,\theta)+g_2(h_3,t_3,\theta)\sin^2 n\theta
 \end{eqnarray}
with $|g_1|\leq Ch_3^{\frac{1}{2}},\ |g_2|\leq Ch_3$ by Lemma
\ref{L09}.
\vskip0.3cm
(ii) Consider $f_2(h,t,\theta)=-\frac{1}{n}xp(t)$ with
$x=\sqrt{\frac{2}{n}}h^{\frac{1}{2}}\cos{n\theta}$.
$$-n\partial_\theta\Big(
 f_2\Big|_{(h_3,t_3,\theta)}\Big)=(x_hh_{2,\theta}+x_\theta)p(t)+xp'(t)t_\theta,$$
and
\begin{eqnarray}\label{}
\nonumber n\partial^2_\theta
 \Big(
 f_2\Big|_{(h_3,t_3,\theta)}\Big)&=&p(t)(x_{hh}(h_{2,\theta})^2+2x_{h\theta}h_{2,\theta}+x_{h}h_{2,\theta\theta}+x_{\theta\theta})\\
\nonumber&&+(x_hh_{2,\theta}+x_\theta)p'(t)t_\theta+xp''(t)t_\theta+xp'(t)t_{\theta\theta}.
 \end{eqnarray}
 By Lemma
\ref{L09} and (\ref{t3}), it holds that
\begin{eqnarray}\label{rr02}
\Bigg|\partial^2_\theta
 \Big(
 f_2\Big|_{(h_3,t_3,\theta)}\Big)\Bigg|\leq Ch_3^{\frac{1}{2}}. \end{eqnarray}

\vskip0.3cm
(iii) Consider $\psi(x)$ with
$x=\sqrt{\frac{2}{n}}h^{\frac{1}{2}}\cos{n\theta}$.
 Similar to part (i),  it holds that $$\partial_\theta\Big(
\psi(x)\Big|_{(h_3,t_3,\theta)}\Big)=\psi'(x)(x_hh_{2,\theta}+x_\theta),$$
and
\begin{eqnarray}\label{rr03}
\nonumber \partial^2_\theta\Big(
 \psi(x)\Big|_{(h_3,t_3,\theta)}\Big)&=&-\psi(x)(x_hh_{2,\theta}+x_\theta)^2\\
\nonumber
&&+\psi'(x)(x_{hh}(h_{2,\theta})^2+2x_{h\theta}h_{2,\theta}+x_{h}h_{2,\theta\theta}+x_{\theta\theta})\\
&=&g_3(h_3,t_3,\theta)+g_4(h_3,t_3,\theta)\sin^2 n\theta
 \end{eqnarray}
with $|g_3|\leq Ch_3^{\frac{1}{2}},\ |g_4|\leq Ch_3$ by Lemma
\ref{L09}.
\vskip0.3cm
(iv) Consider $R_{01}(h,t,\theta)$, $$\partial_\theta\Big(
 R_{01}\Big|_{(h_3,t_3,\theta)}\Big)=R_{01,h}h_{2,\theta}+R_{01,t}t_\theta+R_{01,\theta},$$
and
\begin{eqnarray}\label{}
\nonumber \partial^2_\theta
 \Big(
 R_{01}\Big|_{(h_3,t_3,\theta)}\Big)&=&(R_{01,hh}h_{2,\theta}+R_{01,ht}t_\theta+R_{01,h\theta})h_{2,\theta}\\
\nonumber
&&+(R_{01,th}h_{2,\theta}+R_{01,tt}t_\theta+R_{01,t\theta})t_\theta\\
\nonumber &&+R_{01,\theta h}h_{2,\theta}+R_{01,\theta
t}t_\theta+R_{01,\theta\theta}+R_{01,h}h_{2,\theta\theta}+R_{01,t}t_{\theta\theta}
 \end{eqnarray}
 By Lemmas
\ref{L06}, \ref{L09} and (\ref{t3}), it holds that
\begin{eqnarray}\label{rr04}
\Bigg|\partial^2_\theta
 \Big(
 R_{01}\Big|_{(h_3,t_3,\theta)}\Big)\Bigg|\leq Ch_3^{\frac{1}{2}}. \end{eqnarray}

\vskip0.3cm
(v) Similar to case (iv), we have
\begin{eqnarray}\label{rr05}
\Bigg|\partial^2_\theta
 \Big(
 R_{02}\Big|_{(h_3,t_3,\theta)}\Big)\Bigg|\leq C.
 \end{eqnarray}

From (\ref{rr01})-(\ref{rr05}), we have
\begin{eqnarray}\label{rr06}
\partial^2_\theta
 \Big(
 R\Big|_{(h_3,t_3,\theta)}\Big)=g_{}(h_3,t_3,\theta)+g_6(h_3,t_3,\theta)\sin^2 n\theta
 \end{eqnarray}
with $|g_{}|\leq Ch_3^{\frac{1}{2}},\ |g_6|\leq Ch_3$.
\vskip0.3cm
Finally, with (\ref{h23}) and (\ref{rr06}), (\ref{Q02}) is
proved.
\end{proof}

\end{document}